\documentclass[reqno,12pt]{amsart}
\usepackage{graphicx}
\usepackage{booktabs}

\usepackage[english]{babel}

\usepackage{amsfonts,amssymb,amsthm,amscd,latexsym}
\usepackage{amsmath}
\allowdisplaybreaks
\usepackage{graphicx,epsfig}
\usepackage{psfrag}
\usepackage{perpage}
\usepackage{url}
\usepackage{color}
\usepackage{bbm}
\usepackage{epstopdf}
\usepackage[utf8]{inputenc}
\usepackage[T1]{fontenc}
\usepackage{microtype}
\usepackage{hyperref}
\usepackage{mathabx}
\usepackage{rotating}
\usepackage{amsbsy,enumerate}
\usepackage{graphicx}
\usepackage{ccaption}
\usepackage{comment}
\usepackage{mathrsfs} 
\usepackage{tikz}
\usepackage{fullpage}

\usepackage{amsmath,amsfonts,amssymb,amsthm,amscd,latexsym}
\usepackage{graphicx,epsfig}
\usepackage{psfrag}
\usepackage{perpage}
\usepackage{url}
\usepackage{color}
\usepackage{bbm}
\usepackage{epstopdf}
\usepackage[utf8]{inputenc}
\usepackage[T1]{fontenc}
\usepackage{microtype}
\usepackage{hyperref}
\usepackage{mathabx}

\counterwithin{figure}{section}
\numberwithin{figure}{section}

\numberwithin{equation}{section}

\newtheorem{theorem}{Theorem}[section]
\newtheorem{lemma}[theorem]{Lemma}

\newtheorem{property}[theorem]{Property}

\newcommand{\dd}{{\rm d}}

\newcommand{\cI}{\mathcal{I}}

\newcommand{\cM}{{\mathcal{M}}}

\newcommand{\cN}{{\mathcal{N}}}

\newcommand{\cL}{{\mathcal{L}}}
\newcommand{\cJ}{{\mathcal{J}}}
\newcommand{\cW}{{\mathcal{W}}}

\newcommand{\sS}{_{\rm S}}
\newcommand{\sI}{_{\rm I}}
\newcommand{\sR}{_{\rm R}}
\newcommand{\hS}{^{\rm S}}
\newcommand{\hI}{^{\rm I}}

\newcommand{\rS}{{\rm S}}
\newcommand{\rI}{{\rm I}}
\newcommand{\rR}{{\rm R}}

\newcommand{{\paa}[1]}{p_{00,#1}}
\newcommand{{\pab}[1]}{p_{01,#1}}
\newcommand{{\pba}[1]}{p_{10,#1}}
\newcommand{{\pbb}[1]}{p_{11,#1}}

\newcommand{\eee}{{\rm e}}

\allowdisplaybreaks[4]

\title{Infection models on dense dynamic random graphs}
\author{Simone Baldassarri}
\address{Gran Sasso Science Institute, Viale Francesco Crispi 7, 67100 L’Aquila, Italy}
\email{simone.baldassarri@gssi.it} 

\author{Peter Braunsteins}
\address{School of Mathematics and Statistics, Anita B.\ Lawrence Centre, UNSW Sydney, Sydney NSW 2052, Australia}
\email{p.braunsteins@unsw.edu.au}

\author{Frank den Hollander}
\address{Mathematisch Instituut, Universiteit Leiden, Einsteinweg 55, 2333 CC  Leiden, The Netherlands}
\email{denholla@math.leidenuniv.nl} 

\author{Michel Mandjes}
\address{Mathematisch Instituut, Universiteit Leiden, Einsteinweg 55, 2333 CC  Leiden, The Netherlands}
\email{m.r.h.mandjes@math.leidenuniv.nl} 

\date{\today}

\begin{document}

%%%%%%%%%%%%%%%%%%%%%%%%%%%%%%%%%%%%%%%%%%%%%%%%%%%%%%

\begin{abstract}
    We consider Susceptible-Infected-Recovered (SIR) models on dense dynamic random graphs, in which the joint dynamics of vertices and edges are {\it co-evolutionary}, i.e., they influence each other bidirectionally. In particular, edges appear and disappear over time depending on the states of the two connected vertices, on how long they have been infected, and on the total density of susceptible and infected vertices. Our main results establish functional laws of large numbers for the densities of susceptible, infected, and recovered vertices, jointly with the underlying evolving random graphs in the graphon space. Our results are supported by simulations, which characterize the limiting size of the epidemics, i.e., the limiting density of susceptible vertices, and how the {\it peak} of the epidemics depends on the rate of the evolution of the underlying graph. 
    
    The proofs of our main results rely on the careful construction of a {\it mimicking process}, obtained by approximating the two-way feedback interaction between vertex and edge dynamics with a mean-field type interaction, acting only as one-way feedback, that remains sufficiently close to the original co-evolutionary process. To treat the more general setting in which edge dynamics are affected by the proportions of susceptible and infected individuals, we introduce a methodological extension of existing techniques. We thus show that our model exhibits multiple epidemic peaks -- a phenomenon observed in real-world epidemics -- which can emerge in models that incorporate mutual feedback between vertex and edge dynamics.
    
\\[0.2cm]
{\it Acknowledgment.} The work in this paper was supported by the European Union’s Horizon 2020 research and innovation programme under the Marie Skłodowska-Curie grant agreement no.\ 101034253, and by the NWO Gravitation project NETWORKS under grant no.\ 024.002.003. SB was further supported through “Gruppo Nazionale per l’Analisi Matematica, la Probabilità e le loro Applicazioni” (GNAMPA-INdAM).

\bigskip\noindent
\includegraphics[height=3em]{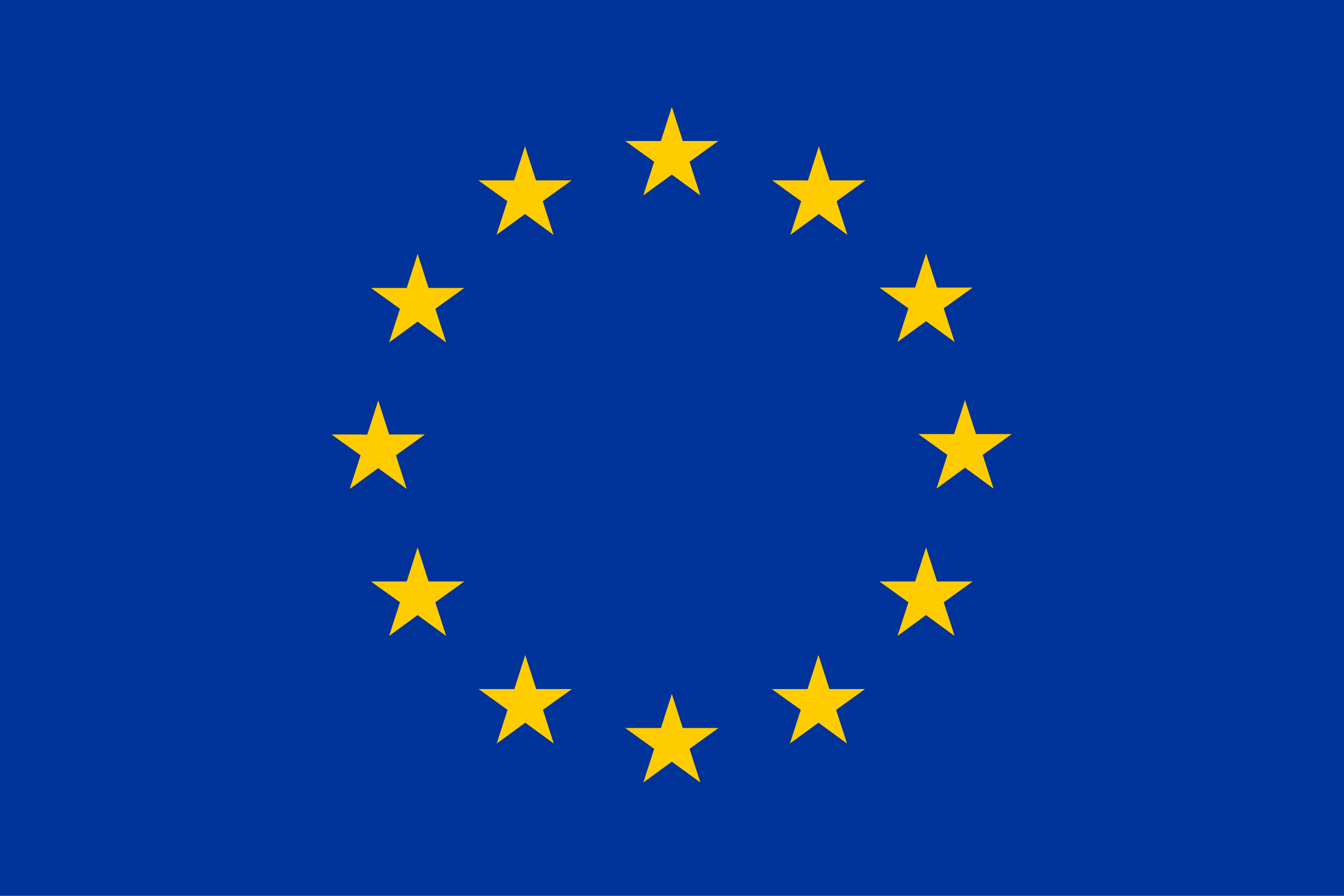} 
\end{abstract}

\maketitle

\newpage

\tableofcontents

%%%%%%%%%%%%%%%%%%%%%%%%%%%%%%%%%%%%%%%%%%%%%%%%%%%%%%%%%%%

\section{Introduction}

Mathematical models play a crucial role in understanding and managing the spread of infectious diseases. These models offer insight into how various parameters influence the dynamics of an epidemic, which can be instrumental in forecasting outbreaks and designing effective control and mitigation strategies.

The base epidemic model is the {\it Susceptible-Infected-Recovered} (SIR) compartmental model. The deterministic version of this model is characterized by a system of differential equations that captures the evolution of the proportion of susceptible, infected, and recovered individuals in the population over time. The same differential equations also appear in the stochastic version of the model, as a large population limit established through a \textit{functional law of large numbers} (FLLN). In the present paper, we establish such a FLLN for a considerably more general version of the base SIR model.

The basic version of the SIR model relies on simplifying assumptions that are typically not satisfied in practice. One key assumption is \emph{homogeneous mixing}, which implies that every individual in the population is equally likely to come into contact with every other individual. If we represent each individual in the population as a vertex in a graph and each social connection through which the epidemic can spread as an edge, then the assumption of homogeneous mixing is equivalent to assuming that \emph{every edge is present in the graph}. However, this assumption overlooks several crucial aspects of real-world social networks during a pandemic. In reality, the network is \textit{heterogeneous}, reflecting that some individuals have more social connections than others. It is also \textit{dynamic}, as connections can appear and disappear over time. Moreover, these dynamics are \textit{co-evolutionary}: the individual epidemiological states (i.e., susceptible, infected, or recovered) are influenced by the structure of the network (since the epidemic spreads through existing connections) and the network itself evolves in response to the epidemic, as individuals may choose to alter their connections to prevent further spread of the disease. 

At a more precise level, a realistic model should incorporate two types of co-evolution:
\begin{enumerate}
\item[(a)] First, the dependence of the network on the individual epidemiological states can be \textit{local}. For example, susceptible individuals may choose to break ties with other individuals they suspect to be infected.
\item[(b)] Second, there can be \textit{global} feedback mechanisms. For instance, quarantine measures may be introduced once the proportion of infected individuals exceeds a certain threshold.
\end{enumerate}
Ideally, we would work with a model in which the underlying graph is heterogeneous (i.e., not complete), is dynamic (i.e., evolves over time), and exhibits co-evolution (i.e., the network influences the epidemic dynamics and is, in turn, shaped by them -- both at local and global scales).

In this paper we introduce a stochastic SIR model that explicitly captures the realistic network features discussed above. We analyze this model in the {\it dense regime} in which the number of edges scales roughly as $n^2$, with $n$ denoting the number of individuals. Our first goal is to establish a FLLN for the proportion of susceptible, infected, and recovered individuals over time. While our model is inherently more complex, this result is in a similar spirit as the FLLN for the base SIR model. In view of the general co-evolutionary dynamics incorporated in the model, we are also interested in how the network evolves during the pandemic. Our second objective is therefore to establish a FLLN for the dynamic network. The proposed model is capable of reproducing complex real-world phenomena, such as the emergence of multiple infection peaks — patterns that cannot be captured in SIR models without incorporating the co-evolutionary feedback between individual states and network structure. We use our results to gain insight into the evolution of the epidemic and the state of the network as the epidemic evolves.

Our results should be viewed within the context of SIR models on random graphs. We begin by providing a brief overview of this subarea. Broadly, the literature can be categorized into four main groups:
\begin{itemize}
\item[$\circ$] A FLLN for the SIR model on a \textit{sparse static} configuration model was proposed in \cite{Volz08}. This result was subsequently rigorously established in \cite{BR13}, \cite{BP12}, \cite{DDMT12}, and \cite{JLW14} under progressively weaker assumptions on the distribution of the vertex degrees in the configuration model. A similar result for an SIR model on a sparse static stochastic block model appeared more recently in \cite{BHI24}. \smallskip
\item[$\circ$] FLLNs for SIR models on \emph{dense static} random graphs have been considered in \cite{DFGTVZ24,KHT22,PPV25}. In these papers, the random graph through which the epidemic spreads is constructed by sampling from a reference graphon. The results in \cite{PPV25} apply when the process exhibits non-Markovian dynamics but hold only when then number of edges scale as $n^2$, whereas the results in \cite{DFGTVZ24,KHT22} consider Markovian dynamics but also hold when the number of edges scale as $n^{1+a}$ for $a \in (0,1]$. \smallskip
\item[$\circ$] FLLNs for the SIR model on \emph{sparse dynamic} random graphs with co-evolutionary feedback are established in \cite{BBLS19, BB22, CHY25, JKYJD19} (see also \cite{DY22} for the SI model). In these papers, susceptible individuals can break edges with infected neighbors, potentially rewiring the edge to another individual. A primary focus of these papers is to establish a discontinuity in the final proportion of susceptible individuals at the critical value of the contact rate. In these papers, the co-evolutionary dynamics are local, in the sense that susceptible individuals break the connection with their infected neighbors in a manner that does not depend on the overall state of the epidemic (i.e., the total number of infected individuals).  \smallskip
\item[$\circ$] A FLLN for the SIR model on \textit{dense dynamic} random graphs is established in \cite{HR24}. In this paper the underlying graph is a dynamic stochastic block model that allows for graph degrees that scale at least as $\log n$ (so at least $n \log n$ edges). However, there is no co-evolutionary feedback in the graph dynamics considered, i.e., the state of the graph evolves independently of the epidemic states of the individuals.
\end{itemize}
Currently, the literature lacks a FLLN for an SIR model on a dense dynamic random graph that incorporates both local and global co-evolutionary feedback. Moreover, no efforts have been made to establish such a law for the evolution of the network itself (i.e., jointly with the proportion of susceptible, infected, and recovered individuals). The present paper aims to fill these gaps. More broadly, our contributions belong to the growing literature for related models on dynamic random graphs. See for instance \cite{ABHdHQ25,BS17,dSOV21,HUVV22,JLM22,LT18,MHZ25,V24} and references therein.

Methodologically, our paper builds on two earlier works \cite{BdHM22,BBdHM24} concerning graph-valued stochastic processes.
A central element of these papers is the concept of a {\it graphon} (developed in \cite{BC08, BC12, L12,LS06,LS07}), which is used to describe the limit of the network. Graphons have been extensively used to characterize the limit of other dense graph-valued processes; see for instance \cite{AdHR25, BK25, BdHM23, C16, G22, R}. However, to the best of our knowledge, the present paper is the first to establish a limit in the space of graphons for a co-evolutionary model with a global feedback mechanism.

In \cite{BdHM22} limits are established for a graph-valued processes with one-way dependence (i.e., not co-evolutionary), in the setting where the edge dynamics can depend on the collective states of all vertices across the network (i.e., global dependence). On the other hand, in \cite{BBdHM24} limits are established for graph-valued processes that are co-evolutionary, but where the co-evolutionary feedback is local only (i.e., no global feedback). In the present paper, we use the mechanism in \cite{BdHM22} to allow the edge dynamics to depend, for example, on the total number of infected individuals (i.e., global feedback), and we adapt the proof technique in \cite{BBdHM24} to establish our results in this more general setting. This technique builds on the notion of a {\it mimicking process} introduced in \cite{BBdHM24}, which approximates the bidirectional interaction between vertex and edge dynamics by a mean-field-type, {\it single-directional} interaction that remains sufficiently close to the original co-evolutionary process. Our analysis suggests that incorporating global feedback mechanisms can be achieved with surprising ease within the framework established in \cite{BBdHM24}. {We expect that this approach is applicable more broadly to other dense co-evolutionary processes.}

The paper is organized as follows. In Section \ref{sec:model_results} we detail the system dynamics, present our main theoretical findings, and provide some discussion, concluding remarks and outlook. The simulations reported in Section \ref{sec:num} reveal that our theory is capable of reproducing patterns observed in practice, most notably trajectories with multiple peaks. A road map of the proof is provided in Section \ref{sec:roadmap}, while the actual proofs are provided in Section \ref{sec:proofs}. 

%%%%%%

\section{Model and main results}\label{sec:model_results}
In this section we first define the stochastic dynamics underlying our SIR model on a dense dynamic random graph. This description specifies the way that we incorporate both local and global co-evolutionary feedback. We then state and briefly discuss our main results. 

%%%%%%

\subsection{Model definition}

In this paper we analyze an SIR model on a dynamic random graph. Given a time horizon $T>0$, we denote by $(G_n(t))_{t\in[0,T]}$ an evolving random graph with vertex set $[n]:=\{1,2,...,n\}$. 
For $i,j \in [n]$, we let $e_{ij}(t)=1$ if edge $ij$ is active at time $t$ and 0 otherwise. The graph $G_n(t)$ is undirected, so $e_{ij}(t)=e_{ji}(t)$, and contains no self-loops, so $e_{ii}(t)=0$. 
Let $x_i(t)\in\{\rS,\rI,\rR\}$ be the state of vertex $i\in[n]$ at time $t\in[0,T]$, where `$\rS$' means that it is susceptible, `$\rI$' means that it is infected, and `$\rR$' means that it is recovered. 

%%%%%%

\subsubsection*{Initialization} At time $t=0$ we initialize the graph as an Erd\H{o}s-R\'enyi random graph with connection probability $p_0\in(0,1)$. Each vertex is infected with probability $q_0\in(0,1)$ and susceptible with probability $1-q_0$, independently of the states of the other vertices and of the initial random graph. Thus, no vertex is in the recovered state at  $t=0$. 

%%%%%%

\subsubsection*{Definition of vertex types}
For any infected vertex $i\in[n]$, we define its {\it infection time} as
\[
t_i\hI := \inf\{s\in[0,T] : x_i(s)=\rI\}.
\]
We define \emph{type} of vertex $i$ at time $t$ as
\begin{equation}\label{eq:deftype}
{y}_i(t) :=
\begin{cases}
    -1 &\hbox{ if } x_i(t) = \rS, \\
    t - t_i\hI &\hbox{ if } x_i(t) = \rI, \\
    T+1 &\hbox{ if } x_i(t) = \rR,
\end{cases}
\end{equation}
where $t - t_i\hI$ is the length of time that vertex $i$ has been infected for. We then define the \textit{empirical type process}, $(F_n(t;\cdot))_{t \in [0,T]}$, through 
\begin{equation}\label{eq:emptypeprocess}
    F_n(t; y) := \dfrac{1}{n} \sum_{i=1}^n \mathbf{1}\{y_i(t) \leq y\}, \quad y \in \mathbb{R}.
\end{equation}

%%%%%%

\subsubsection*{Vertex and edge dynamics}
For each infected vertex $i$, we denote by ${\mathscr I}(\cdot)$ its {\it infectivity}, where ${\mathscr I}:[0,\infty)\to[0,1]$  is a continuous deterministic function.
\begin{itemize}
    \item[$\circ$] \emph{Vertex dynamics.} Let $\cN_i\hI(t)$ be the set of infected neighbors of vertex $i$ at time $t$. At any time $t \in [0,T]$, $x_i(t)$ transitions from $\rS$ to $\rI$ at rate 
    \[
    \frac{\lambda}{n} \sum_{j \in \mathcal{N}_i\hI(t)} {\mathscr I}(y_j(t)),
    \]
    and $x_i(t)$ makes a transition from $\rI$ to $\rR$ at a (normalised) rate of $1$.
    \item[$\circ$] \emph{Edge dynamics.} At any time $t \in [0,T]$, $e_{ij}(t)$ transitions from 0 to 1 (i.e., from inactive to active) at rate
    \[
    \begin{cases}
        \gamma \,\pi_{\rm SS}(F_n(t; \cdot)), \quad &\text{ if  } x_i(t)=x_j(t)=\rS, \\
        \gamma \,\pi_{\rm SI}(y_i(t), F_n(t; \cdot)), \quad &\text{ if  } x_i(t)=\rI, \; x_j(t)=\rS, \\
        \gamma \,\pi_{\rm SI}(y_j(t), F_n(t; \cdot)), \quad &\text{ if  } x_i(t)=\rS, \; x_j(t)=\rI, \\
        \gamma \,\pi_{\rm II}(y_i(t), y_j(t), F_n(t; \cdot)), \quad &\text{ if  } x_i(t)=x_j(t)=\rI,
    \end{cases}
    \]
    and $e_{ij}(t)$ transitions from 1 to 0 (i.e., from active to inactive) at the same rate with $\gamma \,\pi_{s}(\cdot)$ replaced by $\gamma\,(1-\pi_{s}(\cdot))$, with $s\in\{{\rm SS}, {\rm SI}, {\rm II}\}$. When $x_i(t)$ transitions from $\rI$ to $\rR$ we suppose that vertex $i$ resamples adjacent edges with the initial probability $p_0$, after which point these edges are static (i.e., they no longer turn on and off).
\end{itemize}
Note that in our setup the functions $\pi_s{(\cdot)}$ can depend on $F_n(t;\cdot)$, and therefore on the proportion of susceptible, infected and recovered vertices, and also on the distribution of the times which infected vertices have been infected for. We impose a Lipschitz continuity condition on the functions $\pi_{\rm II}{(\cdot)}$. Specifically, we suppose that there exists $L_{\rm II}<\infty$ such that
\[
|\pi_{\rm II}(a_1,b_1,F_1) - \pi_{\rm II}(a_2,b_2,F_2)| \leq L_{\rm II}[|a_1-a_2| +|b_1-b_2|+d_L(F_1,F_2)]
\]
for all $a_1,a_2,b_1,b_2>0$ and all distribution functions $F_1$ and $F_2$, where $d_L$ is the  L\'evy metric defined in \eqref{eq:Levy}. We also impose analogous Lipschitz continuity assumptions on $\pi_{\rm SS}$ and $\pi_{\rm SI}$. 

%%%%%

\subsubsection*{Graphons}

For any $t \in [0,T]$, let $h^{G_n(t)}$ be a function from $[0,1]^2$ to $[0,1]$ which is characterized by
\begin{equation}\label{eqn:empdefn}
h^{G_n(t)}(x,y) := e_{\lceil nx \rceil ,\lceil ny \rceil }(t).
\end{equation}
The function $h^{G_n(t)}$ is referred to as \emph{empirical graphon} associated with $G_n(t)$ and encodes the adjacency matrix of $G_n(t)$ (see Figure \ref{fig:graphon}).

%%%%%%%%%%%%%%%%%%%%%%%%%%%%%%%%%%%%%%%%%%%%%%%%%%%%%%%%%%%%%%%
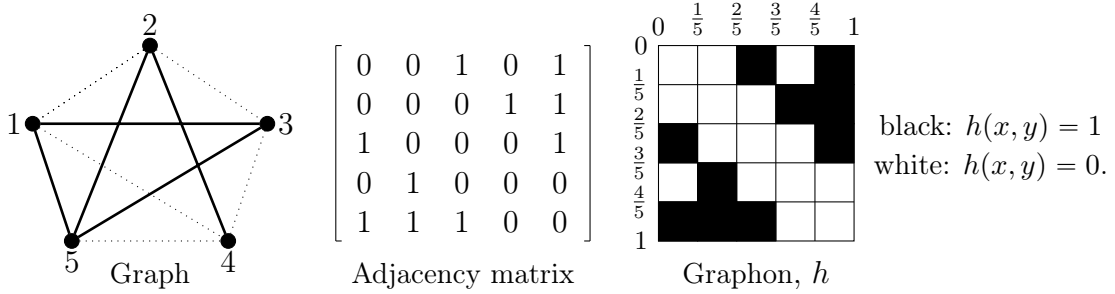
\begin{figure}
\begin{center}
	\begin{tikzpicture}[scale=0.26]
		\path (0,0) node[scale=0.5,draw,shape=circle,fill=black] (p0) {}
		(-6,-4) node[scale=0.5,draw,shape=circle,fill=black] (p1) {}
		(-4,-10) node[scale=0.5,draw,shape=circle,fill=black] (p2) {}
		(4,-10) node[scale=0.5,draw,shape=circle,fill=black] (p3) {}
		(6,-4) node[scale=0.5,draw,shape=circle,fill=black] (p4) {};
		\node[anchor=west, left] at (-6,-4) {{$1$} };
		\node[anchor=west, above] at (0,0) {{$2$} };
		\node[anchor=west, below] at (-4,-10) {{$5$} };
		\node[anchor=west, below] at (4,-10) {{$4$} };
		\node[anchor=west, right] at (6,-4) {{$3$} };
		
		\node[anchor=west, above] at (0,-13) {{\small Graph} };
		\node[anchor=west, above] at (16,-13) {{\small Adjacency matrix} };
		\node[anchor=west, above] at (31,-13) {{\small Graphon}, $h$ };
		\node[anchor=west, above] at (43,-5.5) {\small black: $h(x,y)=1$};
		\node[anchor=west, above] at (43,-7.5) {\small white: $h(x,y)=0$.};
		
		\draw[line width=1pt] 
		(p0) -- (p2)
		(p0) -- (p3)
		(p1) -- (p2)
		(p1) -- (p4)
		(p2) -- (p4);
		\draw[dotted, line width=0.4pt] 
		(p0) -- (p1)
		(p0) -- (p1)
		(p0) -- (p3)
		(p0) -- (p4)
		(p1) -- (p3)
		(p2) -- (p3)
		(p3) -- (p4);
		
		\draw (9.5,0) -- (9.5,-10);
		\draw (22.5,0) -- (22.5,-10);
		\draw (9.5,0) -- (9.8,0);
		\draw (9.5,-10) -- (9.8,-10);
		\draw (22.5,0) -- (22.2,0);
		\draw (22.5,-10) -- (22.2,-10);
		
		\node[anchor=west, right] at (10,-1) {0};
		\node[anchor=west, right] at (12.5,-1) {0};
		\node[anchor=west, right] at (15,-1) {1};
		\node[anchor=west, right] at (17.5,-1) {0};
		\node[anchor=west, right] at (20,-1) {1};
		
		\node[anchor=west, right] at (10,-3) {0};
		\node[anchor=west, right] at (12.5,-3) {0};
		\node[anchor=west, right] at (15,-3) {0};
		\node[anchor=west, right] at (17.5,-3) {1};
		\node[anchor=west, right] at (20,-3) {1};
		
		\node[anchor=west, right] at (10,-5) {1};
		\node[anchor=west, right] at (12.5,-5) {0};
		\node[anchor=west, right] at (15,-5) {0};
		\node[anchor=west, right] at (17.5,-5) {0};
		\node[anchor=west, right] at (20,-5) {1};
		
		\node[anchor=west, right] at (10,-7) {0};
		\node[anchor=west, right] at (12.5,-7) {1};
		\node[anchor=west, right] at (15,-7) {0};
		\node[anchor=west, right] at (17.5,-7) {0};
		\node[anchor=west, right] at (20,-7) {0};
		
		\node[anchor=west, right] at (10,-9) {1};
		\node[anchor=west, right] at (12.5,-9) {1};
		\node[anchor=west, right] at (15,-9) {1};
		\node[anchor=west, right] at (17.5,-9) {0};
		\node[anchor=west, right] at (20,-9) {0};

		\draw [draw=black] (26,-10) rectangle (36,0);
		
		\node[anchor=west, above] at (26,0) {\footnotesize{$0$} };
		\node[anchor=west, above] at (28,0) {\footnotesize{$\frac{1}{5}$} };
		\node[anchor=west, above] at (30,0) {\footnotesize{$\frac{2}{5}$} };
		\node[anchor=west, above] at (32,0) {\footnotesize{$\frac{3}{5}$} };
		\node[anchor=west, above] at (34,0) {\footnotesize{$\frac{4}{5}$} };
		\node[anchor=west, above] at (36,0) {\footnotesize{$1$} };
		
		\node[anchor=west, left] at (26,0) {\footnotesize{$0$} };
		\node[anchor=west, left] at (26,-2) {\footnotesize{$\frac{1}{5}$} };
		\node[anchor=west, left] at (26,-4) {\footnotesize{$\frac{2}{5}$} };
		\node[anchor=west, left] at (26,-6) {\footnotesize{$\frac{3}{5}$} };
		\node[anchor=west, left] at (26,-8) {\footnotesize{$\frac{4}{5}$} };
		\node[anchor=west, left] at (26,-10) {\footnotesize{$1$} };
		
		\draw [fill= black, draw=black] (30,-2) rectangle (32,0);
		\draw [fill= black, draw=black] (34,-2) rectangle (36,0);
		\draw [fill= black, draw=black] (32,-4) rectangle (34,-2);
		\draw [fill= black, draw=black] (34,-4) rectangle (36,-2);
		\draw [fill= black, draw=black] (26,-6) rectangle (28,-4);
		\draw [fill= black, draw=black] (34,-6) rectangle (36,-4);
		\draw [fill= black, draw=black] (28,-8) rectangle (30,-6);
		\draw [fill= black, draw=black] (26,-10) rectangle (28,-8);
		\draw [fill= black, draw=black] (28,-10) rectangle (30,-8);
		\draw [fill= black, draw=black] (30,-10) rectangle (32,-8);
		
		\draw (28,0) -- (28,-10);
		\draw (30,0) -- (30,-10);
		\draw (32,0) -- (32,-10);
		\draw (34,0) -- (34,-10);
		
		\draw (26,-2) -- (36,-2);
		\draw (26,-4) -- (36,-4);
		\draw (26,-6) -- (36,-6);
		\draw (26,-8) -- (36,-8);
		
	\end{tikzpicture} 
\end{center}
\caption{\label{fig:graphon} \small Illustration of the relation between graph, adjacency matrix, and graphon.}
\end{figure}
%%%%%%%%%%%%%%%%%%%%%%%%%%%%%%%%%%%%%%%%%%%%%%%%%%%%%%%%%%

Empirical graphons belong to the space of \textit{graphons} $\mathscr{W}$, containing all functions $h: [0,1]^2 \mapsto [0,1]$ such that $h(x,y)=h(y,x)$ for all $(x,y) \in [0,1]^2$.
We endow $\mathscr{W}$ with the \emph{cut distance}:
\[
d_\square (h_1, h_2) := \sup_{S,T \subseteq [0,1]} \left| \int_{S \times T}{\rm d}x \,{\rm d}y \,[h_1(x,y) - h_2(x,y)] \right|, \quad h_1, h_2 \in \mathscr{W}.
\]
The space $\mathscr{W}$ is used to establish limiting results for random graphs; we establish our main result in the space of $\mathscr{W}$-valued paths.

%%%%%%%%%%%%%%%%%%%%%%%%%%%%%%%%%%%%%%%%%%%%%%%%%%%%%%%%%%%
\begin{figure}
\includegraphics[width=0.32\textwidth]{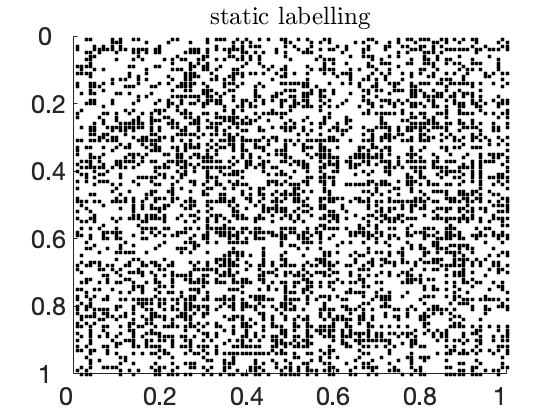}
\includegraphics[width=0.32\textwidth]{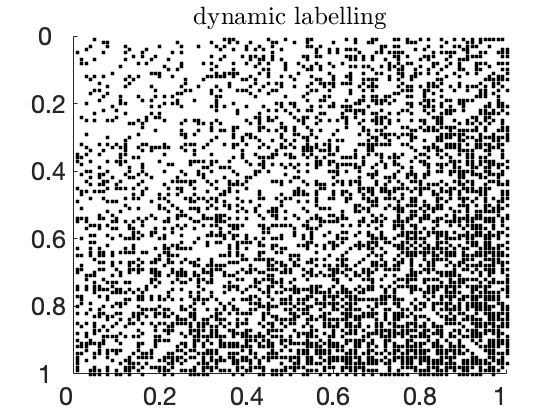}
\includegraphics[width=0.32\textwidth]{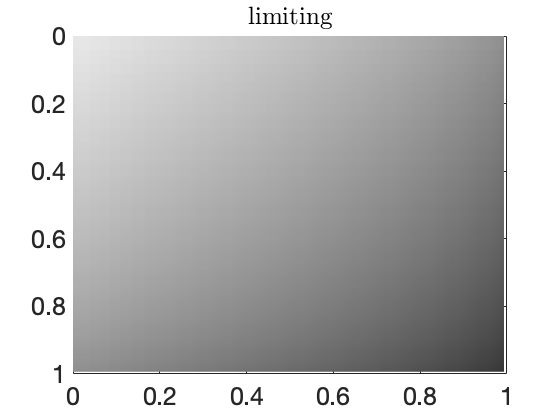}
\caption{\label{fig:label} \small Two graphons with different labellings (left and middle panels), plus the limiting graphon when the labelling of the middle panel has been used (right panel).}
\end{figure}
%%%%%%%%%%%%%%%%%%%%%%%%%%%%%%%%%%%%%%%%%%%%%%%%%

We need to carefully choose how to label the vertices, since $d_\square(h_1, h_2)$ can depend on the labeling convention. For example, in Figure \ref{fig:label} the left and center empirical graphons represent the same graph up to a relabeling of vertices, but only the center graphon is close to the right limiting graphon in the cut distance. 
Throughout this paper we choose to label the vertices such that
\[
y_1(t) \leq y_2(t) \leq \dots \leq y_n(t), \quad  \text{for all $t \in [0,T]$},
\]
with ties broken arbitrarily. This means that susceptible vertices have a lower label than infected vertices, infected vertices have a lower label than recovered vertices, and infected vertices that have been infected for a short time have a lower label than those that have been infected for a long time. This also implies that the labels of the vertices change over time.

%%%%%%

\subsubsection*{Candidate limit} 

In FLLNs, the limiting object can be informally understood as capturing the system's dynamics in the limit of large system size. This deterministic limit, often referred to as the {\it fluid limit}, is typically characterized by a system of differential equations. Below, we outline the procedure for identifying this candidate limit in the context of our system.

We let $(F(t;\cdot))_{t\in[0,T]}$ be a time-varying distribution function which can be thought of as a candidate limit of $(F_n(t;\cdot))_{t \in [0,T]}$ as $n \to \infty$. For any $t \in [0,T]$, $F(t;y)$ is absolutely continuous on $y \in [0,t)$ and has point masses at $-1$, $t$, and $T+1$, with no probability mass elsewhere. We characterize $F$ through differential equations; however, in order to write them in a compact form we require additional notation. First, let $f\sI(t;y):=\frac{\rm d}{{\rm d}y}F(t;y)$ for $y \in [0,t)$, which can be interpreted as the limiting density of the amount of time that infected individuals have been infected for at time $t$, and
\begin{equation}\label{eqn:pdefs}
\begin{aligned}
p\sS (t) &:= F(t;-1) -\lim_{s\uparrow -1} F(t;s), \\
p\sI(t)&:=q_0e^{-t} + \int_0^{t-} {\rm d}y\,f\sI(t;y), \\
p\sR(t)&:=F(t;T+1)-\lim _{s\uparrow T+1}F(t;s),
\end{aligned}
\end{equation}
which can be interpreted as the limiting proportion of susceptible, infected, and recovered individuals at time $t$.

Next, let $H$ be an edge connection probability function, defined by
\begin{align*}
H(t; u &,v, (F(s; \cdot))_{s \in [0,t]}) \\
&:= \mathbb{P}(e_{ij}(t) = 1 \mid y_i(t)=u, y_j(t)=v, (F_n(s;\cdot))_{s \in [0,t]}=(F(s; \cdot))_{s \in [0,t]});
\end{align*}
we remark that this probability is defined more rigorously as an integral of the objects $\pi_{\rm SS}$, $\pi_{\rm SI}$, and $\pi_{\rm II}$ that we introduced in Equation \eqref{eq:H}. Finally let 
\begin{equation}\label{eq:rateI}
\cJ(t) := \mathcal I(t;F(t;\cdot)) 
:= \displaystyle\int_0^t  \dd F(t;u)\,
H(t;-1,u,F(t;\cdot))\, {\mathscr I}(u),
\end{equation}
where $\lambda \cJ(t)$ can be interpreted as the limiting rate (as $n \to \infty$) at which susceptible individuals become infected at time $t$. The function $F$ is characterized by the initial conditions 
\begin{equation}\label{eq:init}
p\sS(0) = 1 - q_0,
\qquad
p\sI(0) = q_0,
\end{equation}
with the initial infection-age distribution concentrated at $u=0$ (that is,
$
f\sI(0;u) = q_0\,\delta_0(u)$),
while the updating equations are
    \begin{align}\label{eq:systPDE}
\dfrac{\partial}{\partial t} p\sS(t)
&= -\lambda \cJ(t)\, p\sS(t), \notag \\
\dfrac{\partial}{\partial t} f\sI(t;u)
+ \dfrac{\partial}{\partial u} f\sI(t;u)
&= - f\sI(t;u),
\qquad u>0, \\
f\sI(t;0)
&= \lambda \cJ(t)\, p\sS(t), \notag \\
\dfrac{\partial}{\partial t} p\sR(t)
&= p\sI(t), \notag
\end{align}
where $f\sI(t;u)$ denotes the density of infected individuals with infection age $u$ at time $t$, and
$p\sI(t) = \smash{\int_0^\infty f\sI(t;u)\,{\rm d}u}$. Due to the initially infected individuals, there is a probability mass at type $t$ for all $t\in[0,T]$, and hence $f_\rI(t;u)$ does not exist when $u=t$. This technical issue can be resolved in a standard way by introducing the measure-valued notation $F_\rI(t;\,{\rm d}u)$ and treating the point mass at $t$ separately. Since this modification is entirely straightforward and does not affect any of the subsequent arguments, we suppress it in the notation to streamline the presentation.

Note that \eqref{eq:systPDE} can be written (in a more involved form) in terms of $F$ and the model parameters defined above, i.e., $p_0$, $q_0$, $\lambda$, $\gamma$, ${\mathscr I}(\cdot)$, $\pi_{\rm SS}(\cdot)$, $\pi_{\rm SI}(\cdot,\cdot)$, and $\pi_{\rm II}(\cdot,\cdot,\cdot)$. Moreover, $F$ can be recovered from the solution to \eqref{eq:systPDE} by letting, for $t\in[0,T]$,
\begin{equation}\label{eq:limitingF}
F(t; y) = \begin{cases}
0, \qquad&y<-1, \\
p\sS(t), \qquad &y \in [-1,0), \\
p\sS(t) + \int_0^y\,{\rm d}u\,f\sI (t;u), \qquad &y\in[0,t), \\
p\sS(t) + p\sI(t), & y\in [t,T+1), \\
1, &y \geq T+1.
\end{cases}
\end{equation}

We also need to construct a candidate limit for the graphon valued process $(h^{G_n(t)})_{t \in [0,T]}$ as $n \to \infty$. To this end, we let $\bar F$ denote the generalized inverse of $F$, which is formally defined as
\[
\bar F(t;x) := \inf\{u : F(t;u) > x\}.
\]
For $t \in [0,T]$, let the candidate limit $g^{[F]}(t;\cdot) \in \mathscr{W}$ be given via
\begin{equation}\label{eq:graphonprocess}
g^{[F]}(t;x,y) := H(t;\bar{F}(t;x),\bar{F}(t;y), (F(t;\cdot)_{t\in[0,T]})).
\end{equation}

%%%%%%

\subsection{Main results}

First we establish that the candidate limit $(F(t;\cdot))_{t\in[0,T]}$ is well-defined. Note that, given \eqref{eq:graphonprocess}, this also implies that $(g^{[F]}(t;\cdot))_{t \in [0,T]}$ is well-defined.

\begin{theorem}\label{thm:PDEsolutions}
There is a unique solution $(F(t;\cdot))_{t \in [0,T]}$ to the system of differential equations characterized by \eqref{eq:systPDE} with initial condition \eqref{eq:init}. In particular, for all $t\in [0,T]$, $F(t;\cdot)$ in \eqref{eq:limitingF} follows from~(i)
\[
p\sS(t) = (1-q_0) \exp\left({-\lambda\int_0^t  \dd s\,\cJ(s)}\right)
\]
and (ii)~$f_I(t;u)$, with $u\in[0,t)$, being the unique solution of the fixed-point equation
\[
f\sI(t;u) = \eee^{-t}f\sI(0;u-t) + \lambda\int_0^t \,\dd s \,\eee^{-(t-s)} \,\cJ(s) \,p\sS(s) .
\]
\end{theorem}

Next we establish convergence of the random processes $(F_n(t;\cdot))_{t \in [0,T]}$ and $(h^{G_n(t)})_{t \in [0,T]}$ to their respective candidate limits $(F(t;\cdot))_{t \in [0,T]}$ and $(g^{[F]}(t;\cdot,\cdot))_{t \in [0,T]}$, as $n\to\infty$. To state this result, we let $(\mathcal{M}, d_L)$ denote the space of probability distributions equipped with the L\'evy metric (which induces the weak topology; see \eqref{eq:Levy}). In the sequel, denote by $D((\mathcal{M},d_L), [0,T])$ (resp.\ $D((\mathscr{W},d_\square), [0,T])$) the Skorokhod space of $\mathcal{M}$-valued (resp.\ $\mathscr{W}$-valued) paths on $[0,T]$. {Here and in the sequel,  `$\Rightarrow$' denotes convergence in distribution. }

\begin{theorem}\label{thm:graphonconv}
    As $n \to \infty$,  on $D((\mathcal{M},d_L),[0,T])$,
    \begin{equation*}
    (F_n(t; \cdot))_{t \in [0,T]}  \Rightarrow (F(t;\cdot))_{t \in [0,T]}
    \end{equation*}
    and, on $D((\mathscr{W},d_\square), [0,T])$,
    \begin{equation*}
        (h^{G_n(t)}(\cdot,\cdot))_{t \in [0,T]}  \Rightarrow (g^{[F]}(t;\cdot,\cdot))_{t \in [0,T]}.
    \end{equation*}
\end{theorem}

\section{Illustrations}\label{sec:num}

In this section, we illustrate our findings. In the first subsection, we focus on the case with only local feedback, presenting expressions for the basic reproduction number $R_0$ and the final fraction of susceptible individuals, and discussing the epidemic peak. The second subsection presents a model with global feedback, which shows that multiple infection peaks can arise.

%%%%%%%

\subsection{Dynamics without global feedback} 

In this subsection, we consider a simpler setting in which global feedback is absent. More specifically, we make the assumption that
\begin{equation}\label{eqn:ngf}
\pi_{\rS\rS}(\cdot)=p_0, \quad \text{and} \quad \pi_{\rS\rI}(u,F) =\pi_{\rS\rI}(u).
\end{equation}
Note that $\pi_{\rS\rS}(\cdot)=p_0$ implies that susceptible individuals are connected with probability $p_0$ {\it at any time $t$}, and $\pi_{\rS\rI}(u,F) =\pi_{\rS\rI}(u)$ entails that there is no global co-evolutionary feedback mechanism.
Under these assumptions, we can follow standard intuitive calculations to derive compact expressions for the basic reproduction number $R_0$ and the final proportion of susceptible individuals $p\sS(\infty)$ as $n$ becomes large {(cf.\ \cite[Section 1]{BR13})}. 

Note that, {under \eqref{eqn:ngf}}, an individual who was infected $u$ time units ago is connected to any susceptible individual with probability
\[ p_{\rS\rI}(u) := p_0\eee^{-\gamma u} + \int_0^u {\rm d}s \,\gamma\eee^{-\gamma s} \,\pi_{\rS\rI}(u-s). \]
The probability that the infected individual infects any susceptible individual throughout their lifetime is then
\[
\bar p_{\rS\rI,n}(u):=\int_0^\infty {\rm d}u  \, p_{\rS\rI}(u) \,\frac{\lambda}{n}\mathscr{I}(u)\,\eee^{-u}.
\]
Considering the expected number of infections caused by a single infected individual among a large number (i.e., $n \to \infty$) of susceptible individuals, we then obtain the basic reproduction number 
\begin{equation}\label{eq:R0}
R_0 = \int_0^\infty {\rm d}u \,  p_{\rS\rI}(u)\, \lambda\mathscr{I}(u)\,\eee^{-u}.
\end{equation}
To write down a formula for the final value $p\sS(\infty)$, note that an initially susceptible individual does not ever get infected with probability $p\sS(\infty)/(1-q_0)$, and if it remains susceptible, then it does not get infected by any of the $n(1-s(\infty))$ individuals that get infected at some point. Consequently
\begin{align*}
\frac{p\sS(\infty)}{1-q_0} &= \lim_{n\to\infty}\left(1-\bar p_{\rS\rI,n}(u)\right)^{n(1-p_sS(\infty))}\\&=\lim_{n \to \infty} \left( 1 - \int_0^\infty {\rm d}u  \, p_{\rS\rI}(u) \,\frac{\lambda}{n}\,\mathscr{I}(u)\,\eee^{-u} \right)^{n(1-p\sS(\infty))} \\
&=\lim_{n \to \infty} \left( 1 - \frac{R_0 (1 -p\sS(\infty))}{n (1-p\sS(\infty))} \right)^{n(1-p\sS(\infty))} =\eee^{-R_0(1-p\sS(\infty))},
\end{align*}
which implies that $p\sS(\infty)$ {is the unique solution in $[0,1]$ of the fixed-point equation}
\begin{equation}\label{eqn:psinf}
p\sS(\infty) = (1-q_0)\,\eee^{-R_0(1-p\sS(\infty))}.
\end{equation}

Let $p_{{\rS},n}(t)$ denote the proportion of susceptible individuals at time $t$ in a system with $n$ individuals. As a consequence of Theorem \ref{thm:graphonconv}, it holds that if $p\sS(\infty)=\lim_{t \to\infty} \lim_{n \to \infty} p_{\rS,n}(t)$, then $p\sS(\infty)$ satisfies \eqref{eqn:psinf}. It should be noted, however, that the limiting proportion of susceptible individuals is most naturally expressed as $p\sS(\infty)=\lim_{n \to \infty}\lim_{t \to\infty} p_{\rS,n}(t)$. The required interchange of limits is generally justified by dominating the epidemic process with a subcritical branching process after some large time $t$ (see, for instance, \cite[Section 6.3]{BHI24}). Here, we encounter two challenges: (1) the vertex types are continuous, which would likely require working with a dominating branching process with a continuous type space, and (2) the dynamic nature of the edges complicates the construction of such a dominating branching process. Addressing these technical challenges is beyond the scope of the present work.

We are also interested in the {\it peak} of the epidemic, defined as  \[i_{\max}:=\max_{t \geq 0} p\sI(t).\] A classical result for the conventional SIR model is that \begin{equation}\label{imax-classical}
i_{\max}=1 - R_0^{-1} + R_0^{-1}\log R_0^{-1}.
\end{equation} 
In our setting, deriving an analogous formula is challenging (see also \cite[Section 3.2]{HR24} for a related analysis). Nevertheless, the findings below illustrate how $\gamma$, which defines the timescale of the edge dynamics, influences $i_{\max}$.

To this end, we first analyze the effect of $\gamma$ on the basic reproduction number $R_0\equiv R_0(\gamma)$, assuming $\mathscr{I}(\cdot)\equiv1$. From \eqref{eq:R0}, by direct computations we find
\[
R_0(\gamma) = \dfrac{\lambda}{\gamma+1}p_0 + \lambda\gamma \int_0^\infty \dd u \,\eee^{-u} \left( \int_0^u \dd s\, \eee^{-\gamma s} \,\pi_{\rS\rI}(u-s) \right) =: A(\gamma) + B(\gamma);
\]
observe that in the static setting (in which $\gamma=0$) this reduces to $R_0(0)=\lambda p_0$. After applying Fubini--Tonelli and the change of variables $v=u-s$ in the inner integral, we can write
\[
B(\gamma) = \lambda\gamma\int_0^\infty \dd s\, \eee^{-\gamma s} \left( \int_0^\infty \dd v \,\eee^{-(s+v)} \,\pi_{\rS\rI}(v) \right)
= \lambda\gamma \,C \int_0^\infty \dd s \,\eee^{-(\gamma+1)s} = \dfrac{\lambda\gamma\, C}{\gamma +1},
\]
where
\[
C:= \int_0^\infty \dd v \,\eee^{-v} \pi_{\rS\rI}(v).
\]
Note that $C<\infty$ because $\pi_{\rS\rI}(\cdot)$ was assumed to be Lipschitz continuous. It thus follows that
\[
R_0(\gamma) = \dfrac{\lambda(p_0+\gamma C)}{\gamma+1},
\]
and as a consequence
\[
R_0'(\gamma) = \dfrac{\lambda(C-p_0)}{(\gamma+1)^2}.
\]
We therefore conclude that $R_0$ is increasing in $\gamma$ if and only if $C>p_0$. 
Note that through Equation \eqref{eqn:psinf} we see that $p_\rS(\infty)$ is monotonically increasing in $R_0$, and hence $p_\rS(\infty)$ is monotonically increasing in $\gamma$ if and only if $C>p_0$. The condition $C>p_0$ is simplest to understand when $\pi_{\rS\rI}(\cdot)$ is constant (say $\pi_{\rS\rI}$), where it reduces to $\pi_{\rS\rI} > p_0$. In this case, once an individual becomes infected, the probability that they are connected to a susceptible individual, $p_{\rS\rI}(u)$, increases with the time since infection $u$. If the edge re-sampling rate $\gamma$ increases, then the increase in $p_{\rS\rI}(u)$ is sped up, leading to an increase in $R_0$ (see \eqref{eq:R0}). Equivalent reasoning also explains why $R_0$ is decreasing in $\gamma$ when $C \equiv\pi_{\rS\rI} < p_0$.

We now turn to the peak of the epidemic, $i_{\rm max}$. First observe that if the classical result \eqref{imax-classical} were to hold for our model, then $i_{\rm max}$ would be increasing in $R_0$. Second, since the number of infected individuals at any given time cannot exceed the total fraction that will eventually be infected, we obtain the upper bound $i_{\rm max} \leq 1 - p\sS(\infty)$, where the right-hand side is likewise increasing in $R_0$.
Taken together with the preceding discussion, these observations suggest that $i_{\rm max}$ is monotonically increasing in $\gamma$ if and only if $C > p_0$. As we are unable to derive a counterpart of \eqref{imax-classical} for the present model, we cannot establish this claim rigorously. Nevertheless, the conclusion is strongly supported by the numerical results shown in Figure~\ref{fig:differentpeaks}.

%%%%%%%%%%%%%%%%%%%%%%%%%%%%%%%%%%%%%%%%%%%%%%%%%%%%%%%
\begin{figure}%[htbp]
\includegraphics[width=7.7cm]{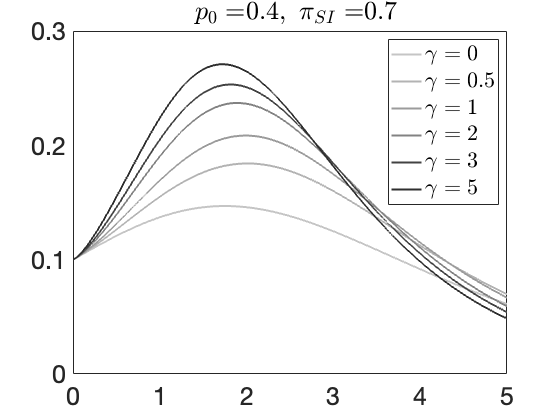}
\includegraphics[width=7.7cm]{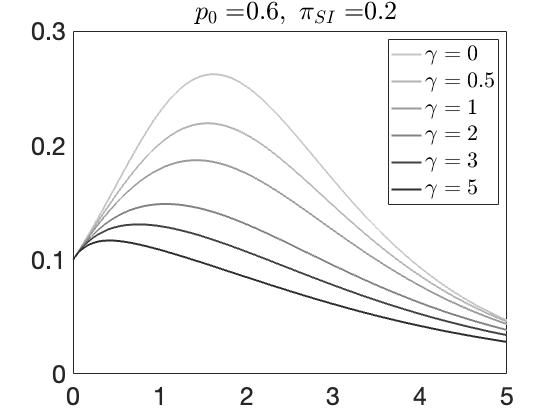}
\caption{\small Each solid curve represents the numerical solution of the system of PDEs for $\pi_{\rS\rS}=p_0$, $\lambda=4$, $q_0=0.1$, and varying $p_0,\pi_{\rS\rI},\gamma$.}
\label{fig:differentpeaks}
\end{figure}
%%%%%%%%%%%%%%%%%%%%%%%%%%%%%%%%%%%%%%%%%%%%%%%%%%%%%%%%

\subsection{Multiple peaks with global feedback}

In this subsection, we present an example that illustrates how global feedback between the epidemic state and network connectivity can produce multiple infection peaks. To capture this behavior, we consider a model in which network connectivity evolves in response to the overall infection level. For $a>0$, we define
\[
\phi(t) := 
\displaystyle \int_0^a \dd s\, f\sI(t; s) ,
\]
which can be equivalently written as
$\phi(t) = 
F_n(t;a)-F_n(t;0).
$
Here, $\phi(t)$ denotes the proportion of individuals who were infected at most $a$ time units ago and who are still infectious at time $t$. It may be interpreted as the population’s perceived overall `threat level'. Since $\phi(t)$ depends on infections that occurred in the past (namely, within the preceding $a$ time units), the perceived threat level generally lags behind the actual threat level. The latter is most naturally quantified by the current rate of infection (see \eqref{eq:rateI} for the limiting rate as $n \to \infty$). For $0<\phi_1<\phi_2<\infty$, define the piecewise linear function
\[
d(\phi) = 
\begin{cases}
0.1 &\hbox{ if } \phi \leq \phi_1, \\
0.1 + 0.8 \dfrac{\phi - \phi_1}{\phi_2-\phi_1} &\hbox{ if } \phi_1<\phi<\phi_2, \\
0.9 &\hbox{ if }  \phi \geq \phi_2,
\end{cases}
\]
which is a behavioral control function, representing how strongly the population is distancing: $d(\phi)\approx 0.9$ (resp.\ $d(\phi)\approx0.1$) means that the system is in `distancing mode' (resp.\ `normal mode'). We then put
\begin{equation}\label{eqn:dpphi}
\begin{array}{ll}
\pi_{\rS\rS}(\phi) &:= (1-d(\phi))\, p_{\rS\rS}^{\rm norm} + d(\phi) \,p_{\rS\rS}^{\rm dist}, \\
\pi_{\rS\rI}(\phi) &:= (1-d(\phi)) \,p_{\rS\rI}^{\rm norm} + d(\phi) \,p_{\rS\rI}^{\rm dist},
\end{array}
\end{equation}
for given numbers ${\smash p_{\rS\rS}^{\rm norm}}$, ${\smash p_{\rS\rS}^{\rm dist}}$, ${\smash p_{\rS\rI}^{\rm norm}}$, and ${\smash p_{\rS\rI}^{\rm dist}}$, where it is envisaged that ${\smash p_{\rS\rS}^{\rm norm}>p_{\rS\rS}^{\rm dist}}$ and  ${\smash p_{\rS\rI}^{\rm norm}>p_{\rS\rI}^{\rm dist}}$. 
We have thus constructed a simplified model of `behavioural response' (e.g., by government interventions or voluntary distancing), enforcing that contact rates drop when people perceive high infection.

In the numerical illustration in Figures \ref{fig:doublepeak} and \ref{fig:functionphi}, we have chosen the parameters: 
\[
\begin{aligned}
p_0 &= 0.1, & q_0 &= 0.05,  & \phi_1 &= 0.24, & \phi_2 &= 0.28, & \gamma &= 20, & \lambda &= 10, \\
p_{\rS\rS}^{\rm norm} &= 0.9, &
p_{\rS\rS}^{\rm dist} &= 0.3, &
p_{\rS\rI}^{\rm norm} &= 0.6, &
p_{\rS\rI}^{\rm dist} &= 0.01, &
\pi_{\rm I\rm I} &= 0.3, & a &= 1 .
\end{aligned}
\]

%%%%%%%%%%%%%%%%%%%%%%%%%%%%%%%%%%%%%%%%%%%%%%%%%%%%%%%%%%
\begin{figure}[htbp]
\includegraphics[width=5.1cm]{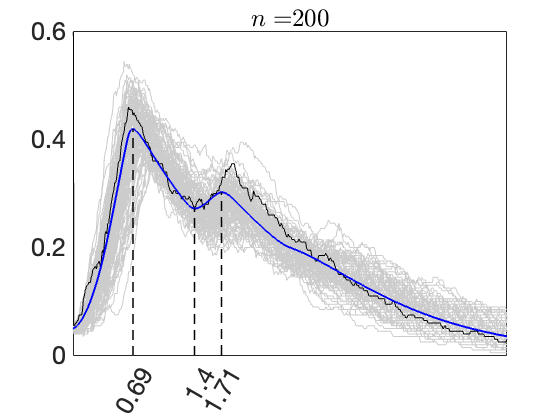}
\includegraphics[width=5.1cm]{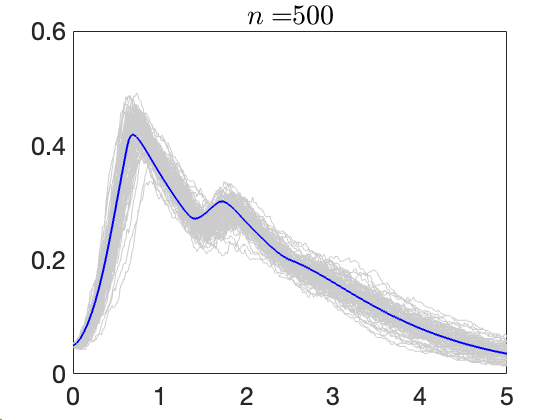}
\includegraphics[width=5.1cm]{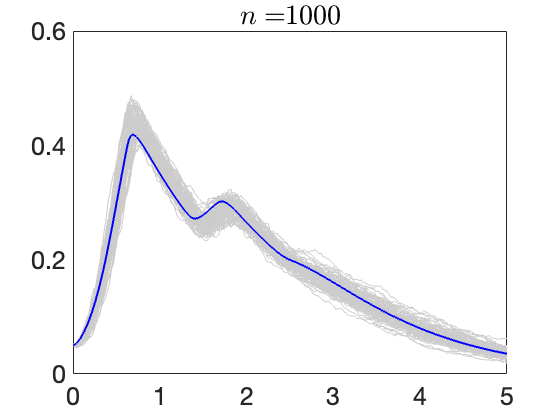}
\includegraphics[width=5.1cm]{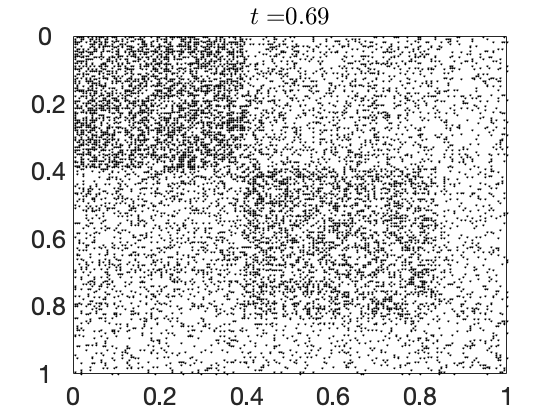}
\includegraphics[width=5.1cm]{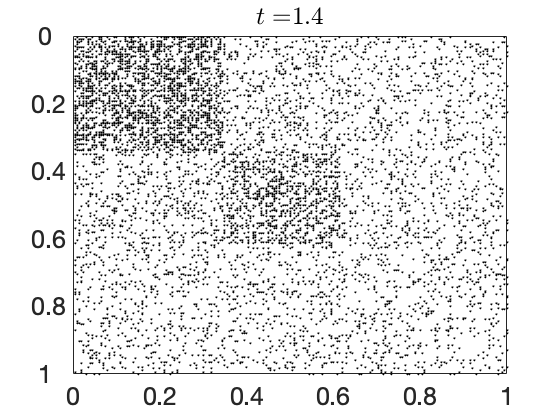}
\includegraphics[width=5.1cm]{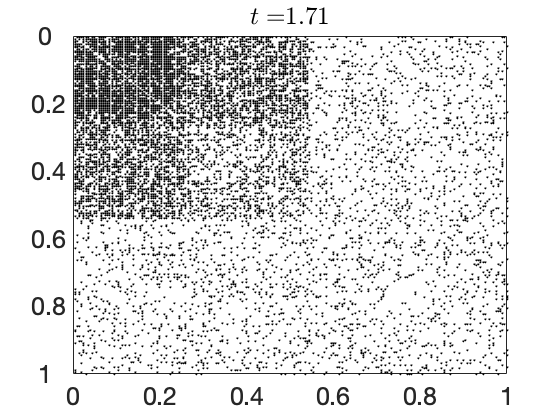}
\includegraphics[width=5.1cm]{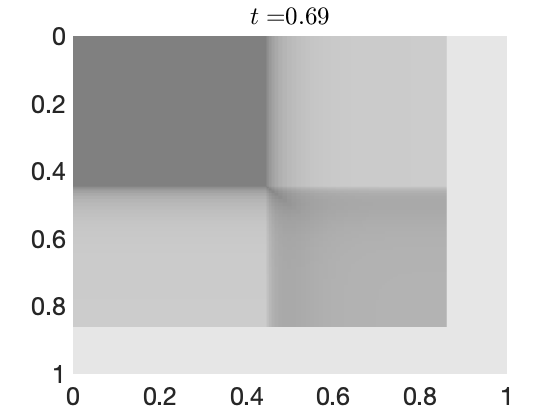}
\includegraphics[width=5.1cm]{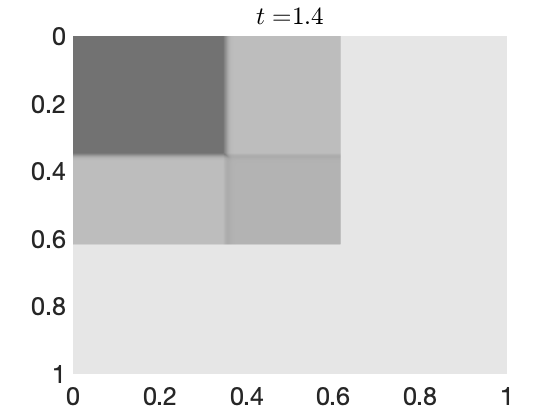}
\includegraphics[width=5.1cm]{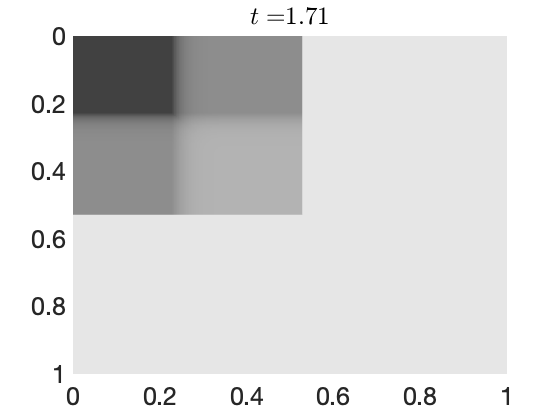}
\includegraphics[width=5.1cm]{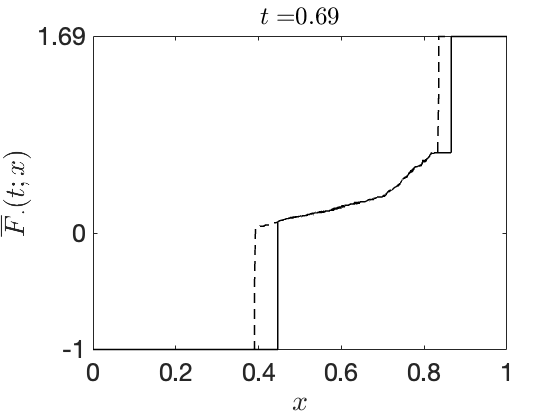}
\includegraphics[width=5.1cm]{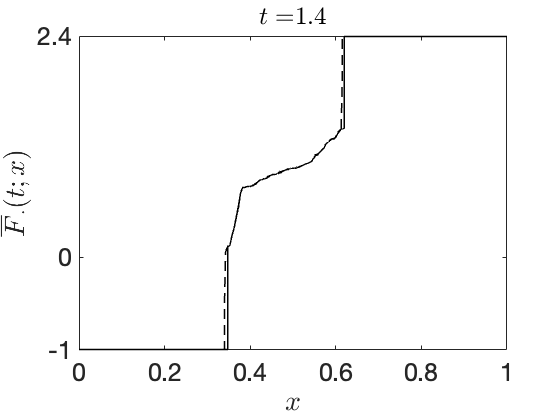}
\includegraphics[width=5.1cm]{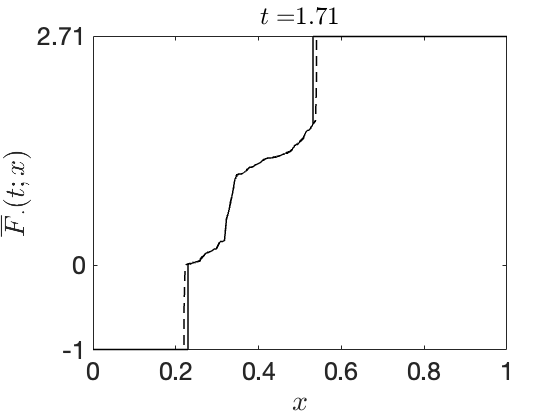}
\caption{\small {The grey trajectories in the top row, from left to right, correspond to simulations of 100 stochastic SIR trajectories of the fraction of infected individuals for $n=200,500,1000$, respectively. The solid blue curves are the numerical solution of the system of PDEs with the same parameters.} {Observe that the `cloud' of simulated trajectories shrinks with $n$, reflecting the convergence to the deterministic limiting path. The second row displays the empirical graphons corresponding to the black trajectory ($n=200$) when $t=0.69$ (first peak), $t=1.4$ (valley between the first two peaks) and $t=1.71$ (second peak); a dot represents an edge.} The labels of the vertices are updated dynamically so that they are ordered lexicographically, i.e., the vertices with state $\rS$ have lower labels than the vertices with state $\rI$, which in turn have lower labels than the vertices with state $\rR$, and then by increasing type. The third row displays the corresponding FLLN. The bottom row displays $\bar{F}_{{200}}(t;x)$ (dashed line) corresponding to the black trajectory and $\bar{F}(t;x)$ (solid line) when $t=0.69,1.4,1.71$ for $x\in[0,1]$.}
\label{fig:doublepeak}
\end{figure}
%%%%%%%%%%%%%%%%%%%%%%%%%%%%%%%%%%%%%%%%%%%%%%%%%%%%%%%

%%%%%%%%%%%%%%%%%%%%%%%%%%%%%%%%%%%%%%%%%%%%%%%%%%%%%%%%%%
\begin{figure}
\includegraphics[width=10cm]{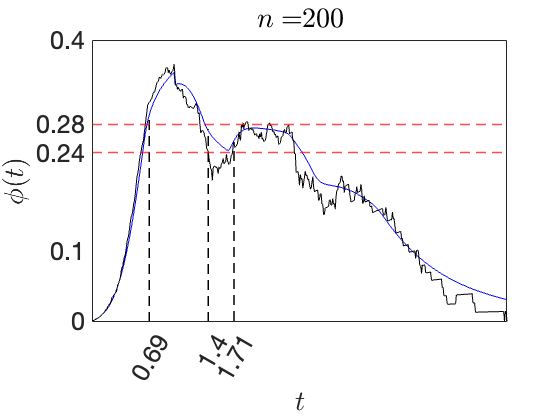}
\caption{\small An illustration of $\phi(t)$ for $t\in[0,5]$: the black and the blue lines correspond to the empirical black trajectory and the smooth blue curve, respectively, of the top left panel in Figure \ref{fig:doublepeak}.}
\label{fig:functionphi}
\end{figure}
%%%%%%%%%%%%%%%%%%%%%%%%%%%%%%%%%%%%%%%%%%%%%%%%%%%%%%%

These figures display the functional law of large numbers for the epidemic states and network that was established in Theorem \ref{thm:graphonconv}, and simulations with finite $n$. The explanation below focuses on the large population limit, and describes how to interpret the figures. The following phases can be distinguished:
\begin{itemize}
    \item[$\circ$] \emph{Initial rise:} In the top row of Figure \ref{fig:doublepeak} we observe that the proportion of infected individuals is increasing for $t \in[0,0.69]$. As displayed in Figure \ref{fig:functionphi}, the process is in normal mode for the vast majority of this time interval. Consequently, there is a high proportion of active SI edges through which the epidemic can spread, and still a relatively high proportion of susceptible individuals which can be infected. Indeed, the initial proportion of susceptible individuals is $1-p_0=0.9$ and, by observing the bottom left plot of Figure \ref{fig:doublepeak}, we see that at $t=0.69$ the proportion of susceptible individuals is just over 0.4 (i.e., the proportion of type $-1$ vertices). 

    \item[$\circ$] \emph{First  peak:} Just before $t=0.69$ the process enters distancing mode, and after $t=0.69$ the proportion of infected individuals starts to decline. In distancing mode, when resampled, the edge probability between $\rS\rI$ individuals is $\pi_{\rS\rI}(\phi)=0.069$, compared to $\pi_{\rS\rI}(\phi)=0.541$ in normal mode (see Equation \eqref{eqn:dpphi}). The adjustment to distancing mode is not instantaneous, however, since edges are resampled at the finite rate $\gamma=20$. Consequently, in the left middle two panels in Figure~\ref{fig:doublepeak}, where we observe the empirical graphon and limiting graphon at $t=0.69$, the graph is losing $\rS\rI$ edges rapidly. From the bottom left plot of Figure~\ref{fig:doublepeak}, we see that the proportions of susceptible and infected individuals are approximately 0.43 and $0.4= 0.83-0.43$ respectively (refer to $x$-axis); hence these $\rS \rI$ edges correspond to $(x,y) \in [0,0.43]\times[0.43,0.83] \; \cup \; [0.43,0.83]\times [0,0.43]$ in the left second and third row of Figure~\ref{fig:doublepeak}.

    \item[$\circ$] {\emph{Decline from first peak and first dip:} Figure \ref{fig:functionphi} demonstrates that even once the proportion of infected individuals starts to decline, the perceived danger level $\phi(t)$ is increasing (i.e., for $t \in [0.69,1]$). This is because $\phi(t)$ is calculated using past infections, and hence the perceived danger level lags behind the actual threat level, i.e., the current rate of infection. This can also be observed after the first dip ($t=1.4$), where the perceived danger level continues to decline even though the proportion of infected individuals is increasing, and after the second peak, where the perceived threat level continues to increase even though the proportion of infected individuals is declining. It is these delayed dynamics that lead to the two infection peaks that we observe in the top row of Figure \ref{fig:doublepeak}.}

    \item[$\circ$] {\emph{Second peak:} At $t=1.71$, the proportion of infected individuals reaches its second peak. In the bottom right of Figure \ref{fig:doublepeak}, for $x \in [0.23, 0.56]$, we can observe the types (i.e., the time since infection) of the currently infected individuals. Here we see that there have been two waves of infection: of those currently infected, a proportion of approximately $0.1=0.33-0.23$ ({refer to $x$-axis}) of the total population were infected in the second wave which occurred between $0$ and $0.31$ time units ago (refer to $y$-axis), approximately $0.04$ were infected during the dip which occurred between $0.31$ and $1.02$ time units ago, and the remaining $0.19=0.56-0.37$ were infected during the first wave which occurred between $1.02$ and $1.71$ time units ago. }

    \item[$\circ$] {\emph{Final decline:} After $t=1.71$ the proportion of infected individuals declines monotonically. This is because there are now very few susceptible individuals left to infect. Indeed, the bottom right of Figure \ref{fig:doublepeak} shows that the proportion of susceptible individuals is just over 0.2 at $t=1.71$. Moreover, at $t=1.71$, the proportion of infected individuals is no longer increasing despite the relatively high proportion of active $\rS\rI$ edges in comparison to $t=0.69$ and $t=1.4$, which can be observed by comparing the plots in the third row of Figure \ref{fig:doublepeak}. The continued decline in the proportion of infected individuals after $t=1.71$ is therefore because the density of $\rS\rI$ edges is never high enough to compensate for the low number of susceptible individuals to infect.}
    
\end{itemize}
%%%%%%%%%%

\section{Road map}\label{sec:roadmap}

In this section, we outline a four-step procedure leading to the proof of our main results. Each step is presented systematically, as this approach may also serve as a useful framework for analyzing related co-evolutionary systems.

The main idea is that we construct a process, referred to as the {\it  mimicking process}, that exhibits {\it one-way dependence}: the edge dynamics depend on the states of the vertices, but not vice versa. This structural simplification makes the mimicking process significantly more tractable than the original co-evolutionary model, which features two-way dependence. In the first three steps of the procedure outlined below we prove our FLLN claims for the mimicking process. In the fourth step we show that the two processes remain `sufficiently close', allowing the FLLN that we establish for the mimicking process to carry over to the co-evolutionary process.

%%%%%%

\smallskip
\noindent {\it Step 1: Define the mimicking process, and establish convergence of its empirical type process.}
A key role in our proofs is played by the mimicking process, a process with one-way dependence that shares the same edge dynamics as the co-evolutionary model but differs in its vertex dynamics. It is expressed via the evolving random graph process $(G_n^*(t))_{t\in[0,T]}$ that corresponds to the empirical type process $(F_n^*(t;\cdot))_{t\in[0,T]}$. For this process, we define the generalized type analogously to the original model, now denoted by  $X_i^*(t)=(x_i^*(t), y_i^*(t))$. In the mimicking process at any time $t\in[0,T]$, $x^*_i(t)$ transitions from state $\rS$ to state $\rI$ at rate $\lambda\,\cJ(t)$, as given in \eqref{eq:rateI}. The infected and recovered vertices behave as in the original model. We define the mimicking process in full detail in Section \ref{sec:coupling}.

We then focus on establishing the convergence of the {empirical type process} of the mimicking process, $(F^*_n(t;\cdot))_{t \in [0,T]}$ as $n\to\infty.$ The following property plays a key role. 

\begin{property}\label{P1}
    The empirical type process satisfies a stochastic-process limit: $F^*_n\Rightarrow F$ as $n\to\infty$ on $D(\cM([0,T]\cup\{T+1\}),[0,T])$. 
\end{property}

To establish this property in our setting,  we characterize $F$ via the generator of the process $(x_i^*(t), y_i^*(t))_{t\in[0,T]}$. It takes the form
\begin{equation}\label{eq:generator}
\begin{array}{ll}
(\cL_t f)(x,y) = \big( \lambda \cJ(t)\,\mathbf{1}_{\{x=\rS\}} + \mathbf{1}_{\{x=\rI\}}\big)[f(x',y)-f(x,y)]+ b(x,y)\,\dfrac{\partial}{\partial y} f(x,y),
\end{array}
\end{equation}
where $x'$ is the state of the selected vertex after changing its state $x$ at time $t$. Here $b(x,y)$ is characterized as follows: (i)~$b(x,y)=0$ for $x\in\{\rS,\rR\}$, because the type of a susceptible or recovered vertex does not change in time if its state does not, and (ii)~$b(\rI,y)=1$. Indeed, if vertex $i$ has state $\rI$ during the entire time interval $(t,t+\dd t)$, then
\[
y_i^*(t+\dd t) = y_i^*(0) + t + \dd t - t_i^\rI  = y_i^*(t) + \dd t,
\]
so that the time derivative of $y_i^*(t)$ equals 1. This gives rise to the Kolmogorov forward equations in \eqref{eq:systPDE} that characterizes $F$, and that verify Property \ref{P1} above.

\noindent 
{\it {Step 2: Express the edge-connection probability, corresponding to the mimicking process, in terms of the types.}} The goal is develop expressions (i)~for the probability that in the mimicking process there is an active edge at time $t$, in terms of the types of vertices $i$ and $j$, and (ii)~for the path of the empirical distribution $(F^*_n(s;\cdot))_{s\in[0,t]}$ up to time $t$. In this context, the key property is the following. 

\begin{property}\label{P2} At any time $t$, edge $ij$ is active with probability 
\[
H(t; y^*_i(t),  y^*_j(t), (F^*_n(s;\cdot))_{s\in[0,t]}),
\]
conditionally independently of all the other edges given $(y^*_i(s), y_j^*(s), F^*_n(s;\cdot))_{s\in[0,t]}$. \end{property}

Given the paths of $x_i^*(\cdot)$ and $x_j^*(\cdot)$, the probability that edge $ij$ is active at time $t$ is
\[
p_{ij}(t) = \eee^{-\gamma t}p_0 + \gamma \int_0^t \dd s \, \eee^{-\gamma s} \pi_{x_i^*(t-s),x_j^*(t-s)}(\cdot)=\eee^{-\gamma t}p_0 + \gamma \int_0^t \dd s \, \eee^{-\gamma (t-s)} \pi_{x_i^*(s),x_j^*(s)}(\cdot).
\]
To verify Property \ref{P2}, we need to show that this probability can be expressed only in terms of the types of vertices $i$ and $j$ at time $t$ (i.e.,  $x^*_i(t)$ and $x^*_j(t) = S$) and $(F^*_n(s;\cdot))_{s\in[0,t]}$. To this end, we first observe that if $x^*_i(t) = x^*_j(t) = \rS$, then 
\[
p_{ij}(t) = \eee^{-\gamma t}p_0 + \gamma \int_0^t \dd s \, \eee^{-\gamma (t-s)} \pi_{\rm SS}(F^*_n(s;\cdot)).
\]
Denote $y^*_{ij,\vee}(t):= y^*_i(t)\vee y_j^*(t)$ and $y^*_{ij,\wedge}(t):= y^*_i(t)\wedge y_j^*(t)$.
If $x^*_i(t) = x^*_j(t) = \rI$, then, using the identities $t_i^\rI  \wedge t_j^\rI  = t- {y}^*_{ij,\vee}(t)$ and $t_i^\rI  \vee t_j^\rI  = t- {y}^*_{ij,\wedge}(t) $, we can write 
\begin{align*}
p_{ij}(t) &= \displaystyle\eee^{-\gamma t}p_0 
+ \gamma \int_{0}^{t_i^\rI  \wedge t_j^\rI } \!\dd s\, 
    \eee^{-\gamma (t-s)} \, \pi_{\mathrm{SS}}(F^*_n(s; \cdot))  \\
&\quad + \gamma \displaystyle\int_{t_i^\rI  \wedge t_j^\rI }^{t_i^\rI  \vee t_j^\rI } \!\dd s\,
    \eee^{-\gamma (t-s)} \, \pi_{\mathrm{SI}}(s - (t_i^\rI  \wedge t_j^\rI ), F^*_n(s; \cdot)) \\
&\quad + \displaystyle\gamma \int_{t_i^\rI  \vee t_j^\rI }^{t} \!\dd s\, 
    \eee^{-\gamma (t-s)} \, \pi_{\mathrm{II}}(s - (t_i^\rI  \wedge t_j^\rI ), s - (t_i^\rI  \vee t_j^\rI ), F^*_n(s; \cdot)) \\
&= \displaystyle\eee^{-\gamma t}p_0 
+ \gamma \int_{0}^{t- {y}^*_{ij,\vee}(t)} \!\dd s\, 
    \eee^{-\gamma (t-s)} \, \pi_{\mathrm{SS}}(F^*_n(s; \cdot)) \\
&\quad + \gamma \displaystyle\int_{t-{y}^*_{ij,\vee}(t)}^{t-y^*_{ij,\wedge}(t) } \!\dd s\, 
    \eee^{-\gamma (t-s)} \, \pi_{\mathrm{SI}}(s-(t-{y}^*_{ij,\vee}(t)), F^*_n(s; \cdot)) \\
&\quad + \gamma \displaystyle\int_{t-t-y^*_{ij,\wedge}(t)}^{t} \!\dd s\, 
    \eee^{-\gamma (t-s)} \,\pi_{\mathrm{II}}(s-(t-{y}^*_{ij,\vee}(t)), s-(t-y^*_{ij,\wedge}(t)), F^*_n(s; \cdot)).
\end{align*}
Similarly, if $x^*_i(t) = \rI$ and $x^*_j(t) = \rS$, then 
\[
\begin{array}{ll}
p_{ij}(t) &= \displaystyle\eee^{-\gamma t}p_0 + \gamma \int_{0}^{t-y^*_i(t)} \dd s \, \eee^{-\gamma (t-s)} \pi_{\rm SS}(F^*_n(s;\cdot)) \\
&\quad + \:\displaystyle\gamma\int_{t-y^*_i(t)}^{t} \dd s \, \eee^{-\gamma (t-s)} \pi_{\rm SI}(s-(t-y^*_i(t)), F^*_n(s;\cdot)).
\end{array}
\]
The case $x^*_i(t) = \rS$ and $x^*_j(t) = \rI$ is fully analogous. Finally, if $x_i(t)=\rR$ or $x_j(t)=\rR$, then evidently $p_{ij}(t) = p_0$.

 In view of the above computations, we can write the probability $p_{ij}(t)$, in terms of $y^*_i(t)$, $y^*_j(t)$ and $(F^*_n(s;\cdot))_{s\in[0,t]}$ only:
\begin{align*}
p_{ij}(t) =&\:
\mathbf{1}{\{y^*_{ij,\vee}(t) = T+1\}}p_0  + \mathbf{1}{\{y^*_{ij,\vee}(t) \neq T+1\}} \,\eee^{-\gamma t}p_0 \:+\\
&\: \mathbf{1}\{y^*_{ij,\vee}(t) = -1 \} \,\gamma \int^t_0 {\rm d}s\,\eee^{-\gamma (t-s)}\pi_{\rm SS}(F_n^*(s; \cdot))\:+ \\
&\: \mathbf{1}\{y^*_{ij,\vee}(t) \in [0,T+1) \}\, \gamma \int^{t-y^*_{ij,\vee}(t)}_0 {\rm d}s\,\eee^{-\gamma(t- s)}\pi_{\rm SS}(F_n^*(s; \cdot))\:+ \\
&\: \mathbf{1}\{y^*_{ij,\wedge}(t) = -1 \} \mathbf{1}\{y^*_{ij,\vee}(t) \in [0,T+1) \} \:\\
&\quad\quad \times \gamma \int_{t-y^*_{ij,\vee}(t)}^t {\rm d}s\,\eee^{-\gamma (t-s)}\pi_{\rm SI}(s-(t-y^*_{ij,\vee}(t)), F_n^*(s; \cdot)) \:+\\
&\: \mathbf{1}\{y^*_{ij,\wedge}(t) \in [0,T+1) \}\mathbf{1}\{y^*_{ij,\vee}(t) \in [0,T+1) \} \\
&\quad\quad \times \gamma \Bigg(\int_{t-(y^*_{ij,\vee}(t))}^{t-(y^*_{ij,\wedge}(t))} {\rm d}s\,\eee^{-\gamma (t-s)}\pi_{\rm SI}(s-(t-y^*_{ij,\vee}(t)), F_n^*(s; \cdot))\:+ \\
&\: \quad \quad \quad\quad \int^{t}_{t-y^*_{ij,\wedge}(t)} {\rm d}s\,\eee^{-\gamma (t-s)}\pi_{\rm II}(s-(t-y^*_{ij,\vee}(t)),s-(t-y^*_{ij,\wedge}(t)), F_n^*(s; \cdot)) \Bigg).
\end{align*}
Consequently, the probability of having an active edge between two vertices with type $u$ and $v$ at time $t$ is given by, with $(u,v)^-:=u\wedge v$ and $(u,v)^+:=u\vee v$,
\begin{align}
H(t; u, v, F^*_n(t; \cdot)) 
&= \mathbf{1}{\{(u,v)^+ = T+1\}}p_0  + \mathbf{1}{\{(u,v)^+\neq T+1\}} \,\eee^{-\gamma t}p_0\:+ \notag \\
    & \mathbf{1}\{(u,v)^+ = -1 \} \,\gamma \int^t_0 {\rm d}s\,\eee^{-\gamma (t-s)}\pi_{\rm SS}(F_n^*(s; \cdot))\:+\notag \\
    & \mathbf{1}\{(u,v)^+) \in [0,T+1) \} \,\gamma \int^{t-(u,v)^+}_0 {\rm d}s\,\eee^{-\gamma(t- s)}\pi_{\rm SS}(F_n^*(s; \cdot)) \notag\\
    & \mathbf{1}\{(u,v)^- = -1 \} \mathbf{1}\{(u,v)^+ \in [0,T+1) \} \:+\notag\\
    &\quad \times \gamma \int_{t- (u,v)^+}^t {\rm d}s\,\eee^{-\gamma (t-s)}\pi_{\rm SI}(s-t+ (u,v)^+, F_n^*(s; \cdot)) \:+\notag\\
& \mathbf{1}\{(u,v)^- \in [0,T+1) \}\mathbf{1}\{(u,v)^+ \in [0,T+1) \} \notag\\
&\quad \times \gamma \Bigg(\int_{t- (u,v)^+}^{t- (u,v)^-} {\rm d}s\,\eee^{-\gamma (t-s)}\pi_{\rm SI}(s-t+ (u,v)^+, F_n^*(s; \cdot)) \:+\notag\\
&\quad \quad \:\:\int^{t}_{t- (u,v)^-} {\rm d}s\,\eee^{-\gamma (t-s)}\pi_{\rm II}(s-t+ (u,v)^+,s-t+ (u,v)^-, F_n^*(s; \cdot)) \Bigg).\label{eq:H}
\end{align}
We have thus verified Property \ref{P2}.

\medskip
\noindent
{\it {Step 3: Establish convergence of the graphon-valued process that corresponds to the mimicking process.}} If $F$ is the limit described in Property \ref{P1}, then the induced reference graphon process (see \eqref{eq:graphonprocess}) is our candidate limit for the graph-valued process $G^*_n(t)$ in the space of graphons. To show this, we use the continuous mapping theorem, for which we need the following property.

\begin{property}
    The map $F\mapsto g^{[F]}$ from $D(\cM([0,T] \cup \{T+1\},[0,T])$ to $D((\mathcal{W},d_{\square}),[0,T])$ is continuous, where $\mathcal{W}$ and $d_{\square}$ are defined in Appendix \ref{appA}.  
\end{property}

\noindent
{\it {Step 4: Extend to the original co-evolutionary model.}} We apply the framework above to the mimicking process, and couple the mimicking process with the co-evolutionary process in such a way that discrepancies during the time interval $[0,T]$ have a sufficiently small probability. This will imply that $h^{G_n(t)}(\cdot, \cdot)\implies g^{[F]}(t;\cdot,\cdot)$ as $n\to\infty$ in the space $D((\mathscr{W},d_\square), [0,T])$

\section{Proofs}\label{sec:proofs}

%%%%%

\subsection{Proof of Theorem \ref{thm:PDEsolutions}} 

The expression for $p\sS(t)$ directly follows from the integration of the first equation in \eqref{eq:systPDE}. We just need to prove that, for any $t\in[0,T]$ and $u\in[-1,T+1]$, there exists a unique solution $f\sI(t,u)$ of the second equation in \eqref{eq:systPDE}. To this end, we first derive the fixed-point equation $f\sI(t,u)$ must satisfy by using the method of characteristics, and then we use the Banach fixed-point theorem to prove existence and uniqueness of the solution.
\medskip

\noindent
{\it Method of characteristics.} The characteristics satisfy $\dd u/\dd t =1$, meaning that $u(t)=u(0) + t$. Along this curve the PDE becomes an ODE for the profile $F(t)=f\sI(t,u(t))$, which reads as
\[
F'(t) + F(t) = \lambda \,\cJ(t) \, p\sS(t),
\]
which is equivalent to
\[
\frac{\dd}{\dd t} \left( \eee^t F(t) \right) = \lambda \eee^t \cJ(t)\, p\sS(t).
\]
This thereby leads to
\[
F(t) = \eee^{-t} F(0) + \lambda \int_0^t \dd s\, \eee^{-(t-s)} \cJ(s) \,p\sS(s), 
\]
which reads as
\[
f\sI(t,u) = \eee^{-t} f\sI(0,u-t) + \lambda \int_0^t \dd s\, \eee^{-(t-s)} \cJ(s) \,p\sS(s). 
\]
Note that the solution is not explicit, but it depends on solving a coupled system along all the characteristics. Indeed, the term $\cJ(s)$ depends on the function $f\sI(s,\cdot)$ itself.

\noindent {\it Contraction property.} Define the space $X=C([0,T]\times[-1,T+1],\mathbb R)$ with the norm
\[
\|f\|_\infty := \sup_{t\in[0,T],\,u\in[-1,T+1]} |f(t,u)|, \qquad f \in X.
\]
Define now the operator $\mathcal T:X\to X$ as
\[
(\mathcal T f)(t,u) := \eee^{-t} f(0,u-t) + \lambda \int_0^t \dd s \,\eee^{-(t-s)} \cJ_f(s)\, p\sS(s),
\]
where $\cJ_f(s) :=  \int_0^s  \dd F_f(s;u)\, H(s;-1,u,F_f(s;\cdot))\, {\mathscr I}(u)$, with $F_f$ defined as in \eqref{eq:limitingF} after replacing $f\sI$ by $f$ (note that also $p\sI(t)$ is obtained after this replacement in \eqref{eqn:pdefs}). A solution $f\sI$ of the starting PDE corresponds then to a fixed point of $f=\mathcal T f$. Letting $f,g\in X$, we find, by using that $\cJ(\cdot)$ is Lipschitz continuous with Lipschitz constant $L$,
\[
\begin{array}{ll}
\|(\mathcal T f)(t,u) - (\mathcal T g)(t,u)\|_\infty &= \displaystyle \sup_{t\in[0,T]} \left|\lambda \displaystyle\int_0^t \dd s\, \eee^{-(t-s)} p\sS(s)\, (\cJ_f(s)-\cJ_g(s))\right| \\
&\leq \lambda L \|f-g\|_\infty \, \displaystyle\sup_{t\in[0,T]} \displaystyle\int_0^t \dd s \, s \eee^{-(t-s)} \\
&\leq \lambda L \|f-g\|_\infty \, \displaystyle\sup_{t\in[0,T]} \displaystyle\int_0^t \dd s \, s, \\
&= K(T) \|f-g\|_\infty,
\end{array}
\]
with $K(T):=\frac{1}{2}\lambda L T^2$, showing that $\mathcal T$ is a contraction for small enough $T$.

\noindent 
{\it Local existence and uniqueness.} First, note that $K(T)>0$. Moreover, if $K(T)<1$, then $\mathcal T$ is a contraction for any $t\in[0,T]$, therefore leading to global existence and uniqueness by the Banach fixed-point theorem. Otherwise, by the same theorem we deduce that there exists a unique solution on $[0,T_0]\times[-1,T+1]$, with $T_0$ the unique positive value such that $K(T_0)=1$.

\noindent
{\it Global existence.} After obtaining the solution on $[0,T_0]$, we use the value at $t=T_0$ as new initial data and repeat the same argument. Since $T<\infty$ and the norm of the solution  can be bounded uniformly in $t$ and $u$ as
\[
\begin{array}{ll}
\displaystyle\sup_{t\in[0,T]}\sup_{u\in[-1,T+1]}|f_\rI(t,u)| &\leq \displaystyle\sup_{u\in[-1,T+1]}|f_\rI(0,u)| + \lambda K(T) < \infty,
\end{array}
\]
we can iterate finitely many times to cover all $t\in[0,T]$. This yields global existence and uniqueness for $(t,u)$ in the whole rectangle $[0,T]\times[-1,T+1]$.

\subsection{Coupling}
\label{sec:coupling}
To prove convergence of the process in the space of graphons, we construct a mimicking process following Steps 1--3 in Section \ref{sec:roadmap} and that is close to the original process in $L_1$ (Step 4).

\medskip
\noindent
{\it Mimicking process.} Suppose that the process $(G_n^*(t))_{t\in[0,T]}$ is characterised by the following dynamics:
\begin{itemize}
    \item $G_n^*(0)$ is an ERRG with connection probability $p_0$.
    \item Each vertex $i$ being susceptible (resp.\ infected) is equipped with an independent rate-$\lambda \cJ^*(t)$ (resp.\ rate-$1$) Poisson clock, with $\cJ^*(t):=\cI(t;f^*\sI(t,\cdot), (F^*(t,\cdot))_{t\in[0,T]})$, where $f^*\sI(t,\cdot) = {\mathbb P}(X_i^*(t)\in(\rI,\cdot))$ and $(F^*(t,\cdot))_{t\in[0,T]}$ is the associated limiting empirical type process. When the clock associated to vertex $i$ rings, the following happens:
    \begin{itemize}
        \item if $x_i(t)=\rS$, then vertex $i$ gets infected;
        \item if $x_i(t)=\rI$, then vertex $i$ recovers and all the adjacent edges to vertex $i$ are resampled according to the initial edge probability $p_0$.
    \end{itemize}
    \item Each edge (not involving any vertex having state $R$) is equipped with an independent rate-$\gamma$ Poisson clock. When the clock associated to edge $ij$ rings, then the edge is active with a probability that depends on the state of the connecting vertices. This probability is defined as follows:
    \begin{itemize}
        \item $\pi_{\rS\rS}(F^*(t,\cdot))$ if both adjacent vertices are susceptible in $G_n^*(t)$;
        \item $\pi_{\rS\rI} (y_i^*(t),F^*(t,\cdot))$ (resp.\ $\pi_{\rS\rI} (y_j^*(t),F^*(t,\cdot))$) if vertex $i$ (resp.\ vertex $j$) is infected and vertex $j$ (resp.\ vertex $i$) is susceptible, where $y_i^*(t)$ denotes the type of vertex $i$ in $G_n^*(t)$ for any $i\in[n]$;
        \item $\pi_{\rI\rI}(y_i^*(t), y_j^*(t),F^*(t,\cdot))$ if both adjacent vertices are infected in $G_n^*(t)$.
    \end{itemize}
\end{itemize}

\begin{lemma}\label{lmm:mimicking}
    The process $\{G_n^*(t))_{t\geq0}\}_{n\in\mathbb{N}}$ satisfies the assumptions of \cite[Theorem 3.10]{BdHM22}, and therefore $h^{G_n^*}\Rightarrow g^{[F^*]}$ as $n\to\infty$ in the space $D((\cW,d_\square),[0,T])$.
\end{lemma}
\begin{proof}
    To prove the claim, it suffices to verify that \cite[Assumptions 3.1, 3.6–3.7, 3.9]{BdHM22} hold. The dynamics of each vertex in the mimicking process are independent of those of the other vertices. Consequently, the first three properties discussed in Section \ref{sec:roadmap} are satisfied. We may therefore follow the same reasoning as in \cite[Section 2.4.1]{BBdHM24} to establish a FLLN for the empirical type process as $n \to \infty$.
\end{proof}

\begin{lemma}\label{lmm:coupling}
    There exists a coupling of $\{(G_n(t))_{t\in[0,T]}\}_{n\in\mathbb{N}}$ and $\{(G^*_n(t))_{t\in[0,T]}\}_{n\in\mathbb{N}}$ such that, for any $\delta>0$,
    \begin{equation}
        \lim_{n\to\infty} \dfrac{1}{n} \log \mathbb{P} \left( ||\tilde{h}^{G_n(t)} - \tilde{h}^{G^*_n(t)}||_{L_1}>\delta \text{ for some } t\in[0,T] \right)=0.
    \end{equation}
\end{lemma}
\begin{proof}
    The claim is proved in three steps. In Step I, we construct a coupling between the original and the mimicking processes. The key idea is to associate susceptible vertices with Poisson clocks whose rates differ between the two processes, and to couple these clocks by simulating a Poisson process of intensity $\lambda$ up to time $t$ and retaining each point in $G_n(t)$ and $G_n^*(t)$ with appropriately chosen probabilities, specified below. The first retained point determines the infection time of the vertex in the corresponding process; if no point is retained up to time $t$, the vertex remains susceptible. The same construction applies to the edge dynamics. A {\it difference} arises when a vertex (resp.\ an edge) is infected (resp.\ active) in one process but susceptible (resp.\ inactive) in the other. Step II characterizes the probability of such differences, and Step III provides precise estimates.
 
    \noindent
    {\it Step I: description of the coupling.} Suppose that $(G_n(t))_{t\geq0}$ and $(G^*_n(t))_{t\geq0}$ are generated in the following manner. Recall that $\cN_i\hI(t)$ is the set of infected neighbors of vertex $i$ at time $t$.
    \begin{itemize}
        \item Each initially infected vertex $i\in[n]$ is assigned the same (coupled) rate-1 Poisson clock in both processes. When the clock associated with vertex $i$ rings, vertex $i$ recovers and all the adjacent edges to vertex $i$ are resampled according to the initial edge probability $p_0$ in both $G_n(t)$ and $G_n^*(t)$.
        \item Each initially susceptible vertex $i\in[n]$ is assigned the same (coupled) rate-$\lambda$ Poisson clock in both processes. When the clock associated to vertex $i$ rings, generate an outcome $U$ drawn from $\text{Unif}(0,1)$ distribution.
        \begin{itemize}
            \item If $U\leq\sum_{j\in \cN_i\hI(t)} {\mathscr I}(y_j)/n$, then vertex $i$ becomes infected, otherwise remains susceptible, in $G_n(t)$.
            \item If $U\leq \cJ^*(t)$, then vertex $i$ becomes infected, otherwise remains susceptible, in $G^*_n(t)$.
        \end{itemize}
     
        \item Assign each edge the same (coupled) rate-$\gamma$ Poisson clock in both processes. When the Poisson clock associated with edge $ij$ rings, generate an outcome $U$ drawn from $\text{Unif}(0,1)$ distribution.
        \begin{itemize}
            \item When $x_i(t)=x_j(t)=\rS$, if $U\leq \pi_{\rS\rS}(F_n(t,\cdot))$, then edge $ij$ is active, otherwise it is inactive, in $G_n(t)$.
            \item When either $x_i(t)=\rS$ and $x_j(t)=\rI$, or $x_i(t)=\rI$ and $x_j(t)=\rS$, if $U\leq \pi_{\rS\rI}(y_j(t), F_n(t,\cdot))$ or $U\leq \pi_{\rS\rI}(y_i(t), F_n(t,\cdot))$, respectively, then edge $ij$ is active, otherwise it is inactive, in $G_n(t)$.
            \item When $x_i(t)=x_j(t)=\rI$, if $U\leq \pi_{\rI\rI}(y_i(t), y_j(t), F_n(t,\cdot))$, then edge $ij$ is active, otherwise it is inactive, in $G_n(t)$.
            \item When $x^*_i(t)=x^*_j(t)=\rS$, if $U\leq \pi_{\rS\rS}(F^*(t,\cdot))$, then edge $ij$ is active, otherwise it is inactive, in $G^*_n(t)$.
             \item When either $x^*_i(t)=\rS$ and $x^*_j(t)=\rI$, or $x^*_i(t)=\rI$ and $x^*_j(t)=\rS$, if $U\leq \pi_{\rS\rI}(y^*_j(t), F^*(t,\cdot))$ or $U\leq \pi_{\rS\rI}(y^*_i(t), F^*(t,\cdot))$, respectively, then edge $ij$ is active, otherwise it is inactive, in $G^*_n(t)$.
            \item When $x^*_i(t)=x^*_j(t)=\rI$, if $U\leq \pi_{\rI\rI}(y^*_i(t), y^*_j(t), F^*(t,\cdot))$, then edge $ij$ is active, otherwise it is inactive, in $G^*_n(t)$.
        \end{itemize}
    \end{itemize}

    \noindent
    {\it Step II: majorization of the $L_1$ distance.} If vertex $i$ is susceptible in both models and the clock associated with vertex $i$ rings, then a {\it difference} is formed when vertex $i$ is susceptible in one process and infected in the other. This happens with probability
    \[
    \Bigg| \sum_{j\in \cN_i\hI(t)} \dfrac{{\mathscr I}(y_j)}{n} - \cJ^*(t)\Bigg|.
    \]
    If edge $ij$ has the same state (active or inactive) in both models and the clock associated with edge $ij$ rings, then a difference is formed with probability depending on the generalised types $X_i(t), X_j(t),X_i^*(t),X_j^*(t)$, and on $F_n(t,\cdot),F^*(t,\cdot)$. For instance, if $x_i(t)=x_i^*(t)=\rS$ and $x_j(t)=x_j^*(t)=\rI$, this probability is equal to 
    \[|\pi_{\rS\rI}(y_j(t), F_n(t,\cdot)) - \pi_{\rS\rI}(y^*_j(t), F^*(t,\cdot))|.
    \]

    \noindent
    {\it Step III: bounding the dominating process.} Our approach is to show that, over a small time window $[t, t+\Delta)$ with $\Delta>0$, the number of differences is stochastically dominated by a random variable with suitable properties. Let 
    \begin{itemize}
        \item $t_\Delta=[t,t+\Delta)$ be the time window,
        \item $N_E\sim\text{Bin}(\binom{n}{2},1-\eee^{-\gamma\Delta})$ be the total number of edge clocks that ring in $t_\Delta$,
        \item $N_V\sim\text{Bin}({\mathcal N}\sS(t),1-\eee^{-\lambda\Delta})$ be the total number of clocks associated with susceptible vertices in both processes at time $t$ that ring in $t_\Delta$, where ${\mathcal N}\sS(t)$ is the number of susceptible vertices at time $t$,
        \item $d_E(t)$ be the total number of differences in the edges between $G_n(t)$ and $G_n^*(t)$,
        \item $d_V(t)$ be the total number of differences in vertex states between $F_n(t;\cdot)$ and $F_n^*(t;\cdot)$, where $F_n$ is defined in \eqref{eq:emptypeprocess} and $F_n^*$ is defined analogously using the graph $G_n^*(t)$. Here, by $F_n(t;\cdot)$ and $F_n^*(t;\cdot)$ we mean the complete information on vertex types and states.
    \end{itemize}
Given $\Omega\equiv(N_E,N_V,d_E(t),d_V(t),F_n(t;\cdot), F^*_n(t;\cdot))$, we want to bound the probability $\Pi_n(\Delta)$ that, when an edge clock (of an edge chosen uniformly at random) rings, a difference is formed during any time $s\in t_\Delta$. Note that if the clock of a vertex pair $ij$ rings, then a difference can only be formed when either $x_i(s)\neq x_i^*(s)$ or $x_j(s)\neq x_j^*(s)$, or there is a difference in the types or in the empirical type distributions $F_n(s;\cdot)$ and $F^*(s;\cdot)$. At any time $t$, we can then partition the set of all the vertex pairs as those connecting at least one vertex with a difference of state between the two models, and those which connect both vertices having the same state in the two models, which we refer to as $E_1(t)$ and $E_2(t)$, respectively. Since each vertex $i$ is part of $n$ vertex pairs, note that
\[
|E_1(t)| \leq n \,(d_V(t) + N_V).
\]

Observing that $d_V(t)=\sum_{i=1}^n \mathbf{1}\{x_i(t)\neq x_i^*(t)\}$, considering the worst case scenario that every time a vertex clock rings during $t_\Delta$ a new difference is formed, and observing that the type of two infected vertices can differ of at most $\Delta$ during $t_\Delta$, we have the following upper bound:
\begin{equation}\label{eq:boundpV}
\begin{aligned}
\Pi_n(\Delta)
&\leq \frac{|E_1(t)|}{\binom{n}{2}}
   + L\, \big(2\Delta + d_L(F_n(s;\cdot), F^*(s;\cdot))\big) \\
&\leq \frac{n(d_V(t)+N_V)}{\binom{n}{2}}
   + L\, \big(2\Delta + d_L(F_n(s;\cdot), F^*(s;\cdot))\big).
\end{aligned}
\end{equation}
where we have used the Lipschitz continuity of the functions $\pi_{\rS\rS},\pi_{\rS\rI},\pi_{\rI\rI}$, with the constant $L$ defined as $\max\{L_{\rS\rS}, L_{\rS\rI}, L_{\rI\rI}\}<\infty$. To control this probability, using the triangular inequality for the Lévy metric (see Appendix \ref{appA}), we can write
\[
 d_L(F_n(s;\cdot), F^*(s;\cdot)) \leq  \underbrace{d_L(F_n(s;\cdot), F_n^*(s;\cdot))}_{(a)} +  \displaystyle \underbrace{d_L(F^*_n(s;\cdot), F^*(s;\cdot))}_{(b)}.
\]
Considering the worst case scenario as explained above, for any $s\in t_\Delta$, the term (a) can be bounded as
\[
d_L(F_n(s,\cdot), F_n^*(s,\cdot)) \leq \dfrac{d_V(t)+N_V}{n},
\]
while the term (b) tends to zero as $n\to\infty$ directly from the Glivenko-Cantelli theorem, because in the mimicking process vertices behave independently and therefore $F^*_n(s;\cdot)$ converges in distribution to $F^*(s;\cdot)$. Thus, \eqref{eq:boundpV} becomes
\[
\begin{aligned}
\Pi_n(\Delta)
&\leq \dfrac{n(d_V(t)+N_V)}{\binom{n}{2}}
   + L \left(2\Delta + \dfrac{d_V(t)+N_V}{n} + o(n) \right) \\
&\leq \left(1 + \frac{L}{2}\right)
   \dfrac{n(d_V(t)+N_V)}{\binom{n}{2}}
   + 2L\Delta + o(n).
\end{aligned}
\]
Because this bound is uniform in $t_\Delta$ and does not depend on $N_E$, given $\Omega$, and the term $o(n)$ can be ignored in the limit as $n\to\infty$, we therefore have
\[
d_E(t+\Delta) - d_E(t) \overset{\rm st}{\leq} \text{Bin} \left( N_E, \left(1 + \frac{L}{2}\right)\dfrac{n(d_V(t)+N_V)}{\binom{n}{2}} + 2L\Delta \right),
\]
where $\overset{\rm st}{\leq}$ means ‘stochastically dominated by'.

Given $\Omega$, we next want to bound the probability that, when a vertex clock (of a susceptible vertex in both processes at time $t$, chosen uniformly at random, that is) rings, a difference is formed during any time $s\in t_\Delta$. In other words, we wish to establish an upper bound on 
\[
\dfrac{1}{n}\sum_{i=1}^n\Bigg| \sum_{j\in \cN_i\hI(s)} \dfrac{{\mathscr I}(y_j)}{n} - \cJ^*(s)\Bigg|,
\]
which applies for all $s\in t_\Delta$. Applying the triangular inequality, we have
\[
\begin{array}{ll}
\displaystyle\dfrac{1}{n}\sum_{i=1}^n\Bigg| \sum_{j\in \cN_i\hI(s)} \dfrac{{\mathscr I}(y_j)}{n} - \cJ^*(s)\Bigg| \\ 
\qquad \qquad \leq 
\underbrace{\displaystyle\dfrac{1}{n}\sum_{i=1}^n\Bigg| \sum_{j\in \cN_i\hI(s)} \dfrac{{\mathscr I}(y_j)}{n} - \sum_{j\in \cN_i^{*\rI}(s)} \dfrac{{\mathscr I}(y_j^*)}{n}\Bigg|}_{(i)} + \underbrace{\displaystyle \displaystyle\dfrac{1}{n}\sum_{i=1}^n\Bigg| \sum_{j\in \cN_i^{*\rI}(s)} \dfrac{{\mathscr I}(y_j^*)}{n} - \cJ^*(s)\Bigg|}_{(ii)},
\end{array}
\]
where $\cN_i^{*\rI}(s)$ (evidently) denotes the set of infected neighbours of vertex $i$ in $G^*_n(s)$. We first bound (i) as follows:
\[
\begin{array}{ll}
\displaystyle\dfrac{1}{n}\sum_{i=1}^n\Bigg| \sum_{j\in \cN_i\hI(s)} \dfrac{{\mathscr I}(y_j)}{n} - \sum_{j\in \cN_i^{*\rI}(s)} \dfrac{{\mathscr I}(y_j^*)}{n}\Bigg| &= \displaystyle\dfrac{1}{n^2}\sum_{i=1}^n\Bigg| \sum_{j\in \cN_i\hI(s)} {\mathscr I}(y_j) - \sum_{j\in \cN_i^{*\rI}(s)} {\mathscr I}(y_j^*)\Bigg| \\
&\leq \dfrac{1}{n^2} (d_E(t) + N_E),
\end{array}
\]
where we used that $\sup_{u\in [0,t+\Delta]} {\mathscr I}(u)=1$.

We next bound (ii). To this end, we use the convergence of the mimicking process in the space of graphons. Let $g_n^*(s;u,v)=h^{G_n^*}(s;u,v)$ denote the empirical graphon at time $s$. In addition, let $r_n^{\rS+\rI}(s)$ and $r^{\rS+\rI}(s)$ denote the sum of the proportions of susceptible and infected vertices in the empirical and limiting graphon, respectively, at time $s$. {Similarly, $r_n^{\rS}(s)$ and $r^{\rS}(s)$ denote the corresponding quantities for susceptible vertices.} Thus, we can write
\[
\dfrac{1}{n}\sum_{j\in \cN_i^{*\rI}(s)} {\mathscr I}(y_j^*) = \int_{r_n\hS(s)}^{r_n^{\rS+\rI}(s)} \dd u \, g_n^*\Bigl(s;\frac{i}{n},u\Bigr) {\mathscr I}(F^*(s;u)),
\]
using the definition of the empirical graphon. In addition, by Lemma \ref{lmm:additional} we have 
\[
\cJ^*(s) = \int_{r\hS(s)}^{r^{\rS+\rI}(s)} \dd y \, g^{[F^*]} (s;y,0) {\mathscr I}(F^*(s;y)).
\]
We can therefore write
\begin{align*}
\frac{1}{n}\sum_{i=1}^n
\Bigg|
\sum_{j\in \cN_i^{*\mathrm{I}}(s)} \frac{{\mathscr I}( y_j^*)}{n}
- \cJ^*(s)
\Bigg|
&= \frac{1}{n} \sum_{i=1}^n
\Bigg|
\int_{r_n^{\mathrm{S}}(s)}^{r_n^{\mathrm{S}+\mathrm{I}}(s)}
\dd u \,
g_n^*\Bigl(s;\frac{i}{n},u\Bigr)
{\mathscr I}(F^*(s;u)) \\
&\qquad
- \int_{r^{\mathrm{S}}(s)}^{r^{\mathrm{S}+\mathrm{I}}(s)}
\dd y \,
g^{[F^*]} (s;y,0)
{\mathscr I}(F^*(s;y))
\Bigg| \\
&= \Bigg|
\int_{r_n^{\mathrm{S}}(s)}^{r_n^{\mathrm{S}+\mathrm{I}}(s)}
\dd u \,
g_n^*(s;0,u)
{\mathscr I}(F^*(s;u)) \\
&\qquad
- \int_{r^{\mathrm{S}}(s)}^{r^{\mathrm{S}+\mathrm{I}}(s)}
\dd y \,
g^{[F^*]} (s;y,0)
{\mathscr I}(F^*(s;y))
\Bigg| \\
&\leq
\bigl|
(r_n^{\mathrm{S}+\mathrm{I}}(s) - r^{\mathrm{S}+\mathrm{I}}(s))
+ (r_n^{\mathrm{S}}(s) - r^{\mathrm{S}}(s))
\bigr|
\\
&\qquad + d_{\square}\!\left(
g_n^*(s;0,\cdot),
g^{[F^*]}(s;0,\cdot)
\right).
\end{align*}

Note that, by the FLLN claimed in Lemma \ref{lmm:mimicking}), as $n\to\infty$ each of these terms tends to zero uniformly in $s$ with high probability. Thus, we conclude that
\[
d_V(t+\Delta) - d_V(t) \overset{\rm st}{\leq} \text{Bin}\left( N_V, \dfrac{1}{n^2} (d_E(t) + N_E) + \epsilon(n) \right),
\]
where $\epsilon(n)$ is a function that can be chosen arbitrarily small with high probability, and can therefore effectively ignored in the limit as $n\to\infty$.

Our plan is now to bound the number of differences formed in the vertices and edges between the two models by considering the time intervals $[0,\Delta], [\Delta, 2\Delta],...,[T-\Delta,T]$ in sequence, and then use the union bound. Without loss of generality, we are tacitly assuming that $T$ is a multiple of $\Delta$. We use the results above which, for $n$ large, can be summarised as follows:
\begin{align*}
 N_E &\overset{\rm st}{\leq} \text{Bin} \left( \binom{n}{2}, \lambda\Delta \right), \\
 N_V &\overset{\rm st}{\leq} \text{Bin} \left( n, \gamma\Delta \right), \\
 d_E(t+\Delta) - d_E(t) &\overset{\rm st}{\leq} \text{Bin} \left( N_E, \left(1 + \frac{L}{2}\right)\dfrac{n(d_V(t)+N_V)}{\binom{n}{2}} + 2L\Delta \right), \\
 d_V(t+\Delta) - d_V(t) &\overset{\rm st}{\leq} \text{Bin} \left( N_V, \dfrac{d_E(t)+N_E}{n^2} \right).
\end{align*}
The proof is then completed by following arguments analogous to those used in the proof of \cite[Lemma 3.6]{BBdHM24}.
\end{proof}

%%%%%

\subsection{Proof of Theorem \ref{thm:graphonconv}}

Theorem \ref{thm:graphonconv} now follows from Lemmas \ref{lmm:mimicking}--\ref{lmm:coupling} and \cite[Theorem 3.10]{BdHM22}.  

%%%%%%%%%%%%%%%%%%%%%%%%%%%%%%%%%%%%%%%%%%%%%%%%%%%%%%%

\appendix

%%%%%%%%%%%%%%%% APPENDIX A %%%%%%%%%%%%%%%%%%%%%%%%%%%%%%

\section{Graphons, Lévy metric and an additional lemma}
\label{appA}

%%%

\medskip\noindent
{\bf Graphons.}
Let $\cW$ be the space of functions $h\colon\,[0,1]^2 \to [0,1]$ such that $h(x,y) = h(y,x)$ for all $(x,y) \in [0,1]^2$, endowed with the {\it cut distance}
\begin{equation}
\label{eq:cutdist}
d_{\square}(h_1,h_2) := \sup_{S,T\subseteq[0,1]}\left|\int_{S \times T} \dd x\, \dd y\, [h_1(x,y)-h_2(x,y)]\right|, \quad h_1,h_2 \in \cW.
\end{equation}
On $\cW$, called the space of graphons, there is a natural equivalence relation $\sim$. Let $\Sigma$ be the space of measure-preserving bijections $\sigma\colon\, [0,1] \to [0,1]$. Then $h_1(x,y) \sim h_2(x,y)$ if $\delta_{\square}(h_1,h_2)=0$, where $\delta_{\square}$ is the \emph{cut metric} defined by 
\begin{equation}
\label{deltam}
\delta_{\square}(\tilde{h}_1,\tilde{h}_2) 
:= \inf _{\sigma_1,\sigma_2 \in \Sigma} d_{\square}(h_1^{\sigma_1}, h_2^{\sigma_2}),
\qquad \tilde{h}_1,\tilde{h}_2 \in \widetilde\cW,
\end{equation}
with $h^\sigma(x,y)=h(\sigma x,\sigma y)$. This equivalence relation yields the quotient space $(\widetilde\cW,\delta_{\square})$, which is compact.

A finite simple undirected graph $G$ on $n$ vertices can be represented as a graphon $h^G\in\cW$ by setting
\begin{equation}
\label{eq:graphon}
h^G(x,y) := \left\{
\begin{array}{ll}
1 &\hbox{if there is an edge between vertex } \lceil nx\rceil \hbox{ and vertex } \lceil ny \rceil, \\
0 &\hbox{otherwise},
\end{array}
\right.
\end{equation}
which referred to as the {\it empirical graphon} associated with $G$, and has a block structure.

%%%

\medskip\noindent
{\bf Lévy metric.} We equip the space $\cM(\mathbb{R}_+)$ with the {\it Lévy metric}, so that for two distribution functions $F,G$ we have
\begin{equation}\label{eq:Levy}
    d_L(F,G) := \inf\{ \epsilon>0 : F(x-\epsilon)-\epsilon \leq G(x) \leq F(x+\epsilon)+\epsilon, \forall \ x\in\mathbb{R} \}
\end{equation}
An important property of the Lévy metric is that it satisfies the triangle inequality.

\begin{lemma}\label{lmm:additional}
    For any $s\in t_\Delta$,
    \[ 
    \cJ^*(s)= \int_{r\hS(s)}^{r^{\rS+\rI}(s)} \dd y \, g^{[F^*]} (s;y,0) {\mathscr I}(F^*(s;y)).
    \]
\end{lemma}

\begin{proof}
By definition, we have
\[
\cJ^*(s)=\mathcal I(s;f\sI^*(s,\cdot),F^*(s,\cdot)) = \displaystyle\int_0^T \dd u \, f\sI^*(s,u) H(s;-1,u,F^*(s,\cdot)) {\mathscr I}(u).
\]
By using the change of variable $u=F^*(s;y)$, we obtain
\[
\begin{aligned}
\cJ^*(s)
&= \int_{r^{\mathrm{S}}(s)}^{r^{\mathrm{S}+\mathrm{I}}(s)}
   \dd y \,
   H\bigl(s;-1,F^*(s;y),F^*(s,\cdot)\bigr)
   {\mathscr I}(F^*(s;y)) \\
&= \int_{r^{\mathrm{S}}(s)}^{r^{\mathrm{S}+\mathrm{I}}(s)}
   \dd y \,
   g^{[F^*]}(s;y,0)
   {\mathscr I}(F^*(s;y)).
\end{aligned}
\]
where in the first equality we used the fact that $\dd y = f\sI^*(s,F^*(s;y)) \, \dd u$. This follows from the derivative of the inverse function, which is well defined because the function $F^*(s;\cdot)$ is right-continuous and non-decreasing. Thus, the derivative is zero only in the intervals in which the function is constant, meaning that there are no infected individuals with those types, and therefore we may think the integral not running on these values. Finally, note that the integral goes from ${r^\rS(s)}$ to ${r^{\rS+\rI}(s)}$ because otherwise ${\mathscr I}(\cdot)\equiv0$ and therefore the integrand is zero.
\end{proof}

\end{document}